\documentclass[11pt]{amsart}
\pdfoutput=1
\usepackage[utf8]{inputenc} 
\usepackage[T1]{fontenc}

\usepackage[american]{babel}
\usepackage{amsmath}
\usepackage{latexsym}
\usepackage{amssymb}
\usepackage{enumerate}
\usepackage{xcolor}
\usepackage{tikz}
\usetikzlibrary{matrix,arrows,decorations.pathmorphing}
\usepackage{tikz-cd}
\usepackage{mathtools}
\usepackage{comment}

\usepackage[capitalize, noabbrev]{cleveref}

\allowdisplaybreaks

\DeclareMathAlphabet\mathbfcal{OMS}{cmsy}{b}{n}
\theoremstyle{plain}
\newtheorem{Thm}[subsection]{Theorem}
\newtheorem{Cor}[subsection]{Corollary}
\newtheorem{Prop}[subsection]{Proposition}
\newtheorem{Lem}[subsection]{Lemma}

\theoremstyle{definition}

\newtheorem{Rem}[subsection]{Remark}
\newtheorem{Def}[subsection]{Definition}

\renewcommand{\phi}{\varphi}
\newcommand{\RR}{\mathbb{R}}
\newcommand{\ZZ}{\mathbb{Z}}
\newcommand{\NN}{\mathbb{N}}
\newcommand{\QQ}{\mathbb{Q}}

\newcommand{\eps}{\varepsilon}
\renewcommand{\emptyset}{\varnothing}
\renewcommand{\setminus}{-}

\newcommand{\larg}{{\operatorname{large}}}
\newcommand{\fin}{{\operatorname{fin}}}
\newcommand{\pr}{\operatorname{pr}}
\newcommand{\inc}{\operatorname{inc}}
\newcommand{\infi}{{\operatorname{inf}}}
\newcommand{\ns}{{}^*\!}
\newcommand{\std}{\operatorname{std}}
\newcommand{\Fin}{{\mathfrak{Fin}}}
\newcommand{\Bnd}{{\mathfrak{Bnd}}}
\newcommand{\Cmp}{{\mathfrak{Cmp}}}
\newcommand{\Uni}{{\mathfrak{Uni}}}
\newcommand{\fS}{{\mathfrak S}}
\newcommand{\Isom}{{\operatorname{Isom}}}

\begin{document}

\title{Polyhedral compactifications, I}
\author{Corina Ciobotaru, Linus Kramer and Petra Schwer}
\date{\today}

\begin{abstract}
	In this work we describe horofunction compactifications of metric spaces and finite dimensional real vector spaces through 
	asymmetric metrics and asymmetric polyhedral norms by means of nonstandard methods, that is, ultrapowers of the spaces at hand.
	The polyhedral compactifications of the vector spaces carry the structure of stratified spaces with the strata indexed by dual faces of the polyhedral unit ball. Explicit neighborhood bases and descriptions of the horofunctions are provided.
\end{abstract}

\maketitle

\section{Introduction}

We introduce a compactification for certain metric spaces, using asymmetric metrics and nonstandard methods.
Our ultimate goal is to study polyhedral
compactifications of Euclidean buildings and related spaces via asymmetric horofunctions \cite{CKS}.
Such compactifications of Euclidean buildings have been studied recently by several authors, using quite different methods.
The work in \cite{Charignon} and \cite{Landvogt} relies heavily on combinatorial properties
of unbounded polytopes or diverging sequences in maximal flats, while \cite{RTW, Werner3} employ
deep arithmetic methods.
Our approach to these compactifications is different, new, and geometric.
Our starting point is Gromov's embedding \cite{Gromov}, which we briefly recall.
If $(X,d)$ is a proper CAT(0) space, with a fixed base point $o\in X$, then we may map
$X$ injectively into the ring of continuous functions on $X$ by assigning to $p\in X$ the function
$x\longmapsto d(p,x)-d(p,o)$. The closure of $X$ in $\subseteq C(X)$ (in the topology
of uniform convergence on compact sets, or, equivalently, in the topology of pointwise convergence)
can be identified with the
visual bordification $X\cup \partial_\infty X$ of $X$ \cite[II.8]{BH}. The first main idea is to replace the metric $d$
in Gromov's construction by a Lipschitz equivalent asymmetric distance function $\delta$.
Stretching the standard terminology slightly, we will call the map
$h_p(x)=\delta(p,x)-\delta(p,o)$ a horofunction centered at $p$. 
We put $\widehat X=\overline{\{h_p\mid p\in X\}}$, and we identify $p\in X$ with 
$h_p$. The new elements that appear in $\widehat X$
may be viewed as horofunctions centered at points 'infinitely far away'.

The traditional approach to these horofunctions is to study divergent sequences in $X$.
This leads to rather complicated combinatorial notions of different types of divergence of sequences in $X$.
We follow a completely different approach, which avoids these sequences altogether.
We use the nonstandard extension $\ns X$ of the metric space $(X,d)$, 
and the nonstandard reals
$\ns\RR$. The nonstandard reals form a real closed, non-archimedean field. It contains thus elements
which are bigger than any real number. In this way we may view the horofunctions in $X$
as (standard parts of) functions $x\longmapsto\delta(p,x)-\delta(p,o)$, where now $p\in\ns X$
is a point which is possibly infinitely far away from $o$. 
We believe that this viewpoint is both intuitive and particularly simple.
The following picture visualizes the situation. The point $p$ is in the nonstandard extension $\ns X$ of $X$
at infinite distance from the basepoint $o$. We put $h_p(x)=\std(\delta(p,x)-\delta(p,o))$.

\begin{center}
	\colorlet{shadecolor}{gray!40}
	\begin{tikzpicture}[scale=0.5]
		\draw[shading=radial,inner color=shadecolor,style=dashed] (0,0) circle(2);
		\node at (3,-1) {$X$};
		\draw (0,0) arc(0:18:8);
		\draw (0,0) arc(0:-18:8);
		\node at (-1,-3) {$h_p=0$};
		\node at (-8,0.5) {$p$};
		\draw[fill=black] (-8,0) circle(0.05);
		\node at (0.5,0.5) {$o$};
		\draw[fill=black] (0,0) circle(0.05);
	\end{tikzpicture}
\end{center}

Several things can be seen from this picture. Firstly,
the horofunction $h_p$ is only the restriction of the standard part of $x\longmapsto\delta(p,x)-\delta(p,o)$ to $X$.
Therefore we should expect that different points $p,q\in\ns X$ may yield the same horofunction on $X$.
This leads to an equivalence relation on $\ns X$ which needs to be studied. It corresponds, roughly speaking,
to the different types of divergence of sequences in $X$.
Secondly, we note that we could rescale the metric so that the distance between $p$ and $o$ becomes $1$.
This rescaling process shrinks $X$ to an infinitesimally small spot. 
If the geometry of $X$ is scale-invariant, then the large scale shape of the level set
$h_p=0$ is determined by the infinitesimal structure of the unit spheres of the original distance function.
If $V=X$ is a finite dimensional real vector space and if $\delta$ is given by a (possibly asymmetric) norm
on $V$,
then there are two instances where this infinitesimal structure is accessible: if the unit sphere $S=\{v\in V\mid h_o(v)=1\}$
is a smooth hypersurface, then the infinitesimal structure is given by the tangent hyperplanes. In this case,
the level set $h_p=0$, for $p$ at infinite distance from $o$, is a linear hyperplane and $h_p$ is a linear functional on $V$.
Then $\widehat V$  can be identified with the dual unit ball, in the dual vector space $V^\vee$ of $V$.
The second case where the infinitesimal structure is accessible is the case where the unit ball $B=\{v\in V\mid h_o(v)\leq1\}$ is
a convex compact polyhedron. This case is analyzed in the present article.
The following theorem is one of the main results.

\textbf{Theorem.}
{\em The compactification $\widehat{V}$ of a finite dimensional real vector space $V$ with respect to an asymmetric polyhedral norm $\nu$ is a stratified space, where the strata are indexed by the dual faces of the polyhedral unit ball $B$.
The combinatorial structure of the stratification of $\widehat{V}$ (with respect to the closure relation) is isomorphic to the poset of all faces of the dual polyhedron $B^\vee$ of $B$.
Moreover, $\widehat V$ is homeomorphic to the dual polyhedron $B^\vee$ by a homeomorphisms that preserves the stratification.}

Since Euclidean buildings are composed of copies of Euclidean space, this analysis of polyhedral compactifications
of finite dimensional real vector spaces is crucial for understanding polyhedral compactifications of Euclidean buildings.

This article is organized as follows. In Section \ref{sec:topology} we review the relevant topologies on the space $C(X)$ of continuous
real functions on a metric space $X$. In Section \ref{sec:metrics-norms} we review asymmetric metrics and asymmetric norms. We do this in the setting
of metric spaces, which allows us to bypass all questions about backward and forward topologies which may arise
otherwise. In Section \ref{sec:asymmetric-bordification} we adapt Gromov's embedding to the case of asymmetric metrics. We do this with some care,
since Gromov's result is very often misstated in the literature (even in \cite{BH}). The condition that $X$ is in one
or the other way geodesic is very often omitted, although it is essential in the proof. In Section \ref{sec:5} we introduce the
necessary facts about ultraproducts. We then show that every horofunction arises as some $h_p$, as described above.
If the metric space $X$ is proper, then conversely every function $h_p$ appears in the bordification.
In Section \ref{sec:polyhedral-norms} we put everything together to determine the polyhedral compactification of a finite dimensional vector space $V$.
Let $B$ denote the unit ball of the asymmetric norm on $\RR^m$, with its polyhedral combinatorial structure.
We show that this combinatorial structure determines the equivalence relation mentioned above
(when is $h_p=h_q$?). We obtain a stratification of $\widehat V$ by subsets homeomorphic to affine spaces,
whose combinatorics is encoded by the dual polytope $B^\vee$ of $B$. The last step is the construction of
an explicit homeomorphism $\widehat V\longrightarrow B^\vee$ which preserves the stratification.


Some comments on the history may be in order. 
During the last years, horofunction compactifications of metric spaces and in particular of Riemannian symmetric spaces have been studied by various authors.
We just mention
\cite{Brill},  \cite{Gutierrez}, \cite{HSWW}, \cite{JiSchilling16,JiSchilling}, \cite{KMN}, \cite{KaLe}, \cite{Walsh},
and the excellent books \cite{BJ,GJT}.
The idea to use Gromov's embedding for the construction of
polyhedral compactifications of buildings via asymmetric polyhedral metrics occurred to us in 2012. At about the same time, this
approach was developed independently for Riemannian symmetric spaces by
Kapovich and Leeb \cite{KaLe}, and independently by Haettel, Schilling, Wienhard and Walsh \cite{HSWW}. 
We then discovered that Brill, a student of Behr, had already
used this idea in his 2006 PhD thesis \cite{Brill}. This thesis, which is very concise, lacks just one ingredient:
Brill only considers symmetric metrics.
Our use of nonstandard methods in the context of horofunctions is, to the best of our knowledge, new.

In subsequent work \cite{CKS} we will study horofunctions on an affine building $X$  using the ultrapower $\ns X$ of the building. We note that $\ns X$ is a so-called 
$\Lambda$-building, as studied in \cite{Bennett,KleinerLeeb,BaseChange,KramerWeiss}, axiomatized in \cite{BennettSchwer} and studied by Schwer in her PhD thesis \cite{H-Thesis}.
We will also compare these compactifications in \cite{CKS} with the compactifications constructed in \cite{Landvogt,RTW,Werner3}.
Finally, we will study the dynamics of discrete group actions on the building, using the compactifications. 

\subsection*{Acknowledgement:}

LK and PS were partially supported by SPP2026 \emph{Geometry at Infinity}, project no. 20.  
LK  benefited very much from a stay at the Mathematical Institute Oberwolfach in the winter of 2020. CC  was partially supported by the European Union’s Horizon 2020 research and innovation programme under the Marie Sklodowska-Curie grant agreement No 754513 and The Aarhus University Research Foundation. 
The authors would like to thank Andreas Berning, Siegfried Echterhoff, 
Gaiane Panina and the referee for helpful remarks.


\section{Topologies on function spaces}\label{sec:topology}

In this section we review some material on topologies on function spaces and state Ascoli's theorem. 

Let $(X,d)$ be a metric space. Given $p\in X$ and $\eps\geq0$, we put
\[
B_\eps(p)=\{q\in X\mid d(p,q)<\eps\}\quad\text{ and }\quad
\bar B_\eps(p)=\{q\in X\mid d(p,q)\leq\eps\}.
\]
We call $(X,d)$ \emph{proper} if every closed bounded set $K\subseteq X$ is compact.
Proper metric spaces are always complete.
We recall some basic facts about function spaces.
Let $C(X)$ denote the commutative $\RR$-algebra of all real-valued continuous functions on $X$. There are several topologies on $C(X)$ and related $\RR$-algebras which we briefly review.
We will be interested in the \emph{topology of uniform convergence on compact sets}, the 
\emph{topology of pointwise
	convergence} and the \emph{topology of uniform convergence on bounded sets}. There is a uniform
way to construct these topologies which goes as follows, see \cite{Schaefer,Kelley}.

\medskip
Let $\fS$ be a collection of subsets
of the metric space $X$. 
Assume that 
\begin{enumerate}[(i)]
	\item $X=\bigcup \fS$ and
	\item for all $P,Q\in \fS$, there exists $R\in \fS$ with $P\cup Q\subseteq R$.
\end{enumerate}
Let $C_{\fS}(X)$ denote the vector space of all
real functions $\phi$ on $X$ which are bounded and continuous on every member $Q$ of $\fS$.
Given $Q\in\fS$, $\eps>0$ and $\phi\in C_{\fS}(X)$ we define 
\[
N_{Q,\eps}(\phi)=\{\psi\in C_{\fS}(X)\mid | \psi(q)-\phi(q)|<\eps\text{ for all }q\in Q\}.
\]

These sets form neighborhood bases for a topology $\mathcal T_\fS$ on $C_\fS(X)$, see \cite[III.3]{Schaefer}.
A set $U\subseteq C_{\fS}(X)$ is open if for every $\phi\in U$,
there exist $Q\in\fS$ and $\eps>0$ such that $N_{Q,\eps}(\phi)\subseteq U$. 
In this topology, $C_{\fS}(X)$ becomes a locally convex topological vector space and a commutative topological $\RR$-algebra.
The cases of interest to us are the following.

\begin{enumerate}[\rm(1)]
	\item $\fS=\Fin$ is the collection of all finite subsets of $X$. This yields the topology
	$\mathcal T_\Fin$ of pointwise convergence on $C_\Fin(X)=X^\RR=\prod_X\RR$, which coincides with the 
	product topology. This topology does not depend on the metric $d$. 
	
	\item $\fS=\Uni=2^X$ is the collection of all subsets of $X$. Then we obtain the topology 
	$\mathcal T_\Uni$ of uniform convergence
	on the space $C_\Uni(X)=BC(X)$ of all bounded continuous functions on $X$, and 
	$BC(X)$ is a Banach space. This space is not relevant for us, 
	and $\Uni$ will not be considered below.
	
	\item For $\fS=\Cmp=\{K\subseteq X\mid K\text{ is compact}\}$ we obtain on $C_\Cmp(X)$ the topology
	$\mathcal T_\Cmp$ of uniform convergence on compact sets, which coincides with the compact-open topology.
	Since $X$ is a metric space, a function on $X$ is continuous if and only if its restriction to every
	compact subset of $X$ is continuous, see \cite[VI.8.3 and XI.9.3]{Dugundji}.
	Thus $C_\Cmp(X)=C(X)$ is a complete locally convex topological vector space.
	
	\item If $\fS=\Bnd$ is the collection of all bounded sets, then
	$\mathcal T_\Bnd$ is the topology of uniform convergence on bounded sets, and 
	$C_\Bnd(X)\subseteq C(X)$ is the space 
	of all continuous functions which are bounded on all bounded sets. 
	If we fix a base point $o\in X$, then the set $\{\bar B_{2^k}(o)\mid k\in\NN\}$ is 
	cofinal in $\Bnd$ and $\{N_{\bar B_{2^k}(o),2^{-\ell}}(\phi)\mid k,\ell\in\NN\}$ is a countable 
	neighborhood basis of $\phi\in C_\Bnd(X)$. In particular, we may work with sequences in this space if we wish so.
\end{enumerate}
We note that the natural maps
\[
(C_\Fin(X),\mathcal T_\Fin)\longleftarrow
(C_\Cmp(X),\mathcal T_\Cmp)\longleftarrow
(C_\Bnd(X),\mathcal T_\Bnd)\longleftarrow
(C_\Uni(X),\mathcal T_\Uni)
\]
are continuous, because we have inclusions
\[\Fin\subseteq\Cmp\subseteq\Bnd\subseteq\Uni.\] 
If $X$ is proper, then $\mathcal T_\Bnd=\mathcal T_\Cmp$ and if $X$ is discrete, then $\mathcal T_\Fin=\mathcal T_\Cmp$.
In general, all four spaces and topologies are different.
Now let $o\in X$ be a base point and put 
\[
I_{\fS,o}(X)=\{\phi\in C_{\fS}(X)\mid\phi(o)=0\}.
\]
This is the kernel of the evaluation map at $o$,
\[
C_\fS(X)\longrightarrow\RR, \quad \phi\longmapsto \phi(o),
\]
and hence a maximal ideal in the ring $C_{\fS}(X)$.
The evaluation map at $o$ is continuous, hence $I_{\fS,o}(X)\subseteq C_\fS(X)$ is a closed hyperplane. 
There is a continuous linear projector 
\[
\pr_o:C_{\fS}(X)\longmapsto I_{\fS,o}(X),\quad \phi\longmapsto \phi-\phi(o).
\]
The kernel of $\pr_o$ is the subring of $C_\fS(X)$ consisting of all constant
real functions on $X$, which we identify with $\RR$. 
Therefore $C_{\fS}(X)$ splits as a topological vector space as
\[
C_{\fS}(X)=\RR\oplus I_{\fS,o}(X).
\]
To see this, we note that the natural homomorphism $\RR\oplus I_{\fS,o}(X)\longrightarrow C_\fS(X)$
is continuous and bijective. Its inverse is the map $\phi\longmapsto(\phi(o),\pr_o(\phi))$, which is also 
continuous.

It follows from the diagram 
\[
\begin{tikzcd}
C_{\fS}(X) \arrow{r}[yshift=.5ex]{\inc} \arrow{d} & I_{\fS,o}(X) \arrow[shift right]{l}[yshift=-.6ex]{\pr_o}\\
C_{\fS}(X)/\RR \arrow{ru}
\end{tikzcd}
\]
that there is a natural isomorphism of topological vector spaces
\[
C_{\fS}(X)/\RR\cong I_{\fS,o}(X)
\]
that maps $\phi+\RR$ to $\phi-\phi(o)$. In particular, there is an isomorphism of topological vector
spaces
\[
I_{\fS,o}(X)\cong I_{\fS,p}(X)
\]
for all $o,p\in X$.
We also recall Ascoli's Theorem.
\begin{Thm}[Ascoli's Theorem]\label{Ascoli}
	Assume $F\subseteq C_\Cmp(X)$ is equicontinuous and that 
	for each $p\in X$ the set $F(p)=\{\phi(p)\mid\phi\in F\}\subseteq\RR$ is bounded. Then 
	$F$ has compact closure $\overline F$ in $C_\Cmp(X)$ with respect to the compact-open topology $\mathcal T_\Cmp$. \newline
	This closure $\overline F$ coincides (set-theoretically and topologically) with the closure of $F$ in $C_\Fin(X)=X^\RR$ with respect to the topology $\mathcal T_\Fin$. 
\end{Thm}
\begin{proof}
	The first claim is classical, see \cite[XII.6.4]{Dugundji}. 
	Since the closure $\overline F\subseteq C_\Cmp(X)$ is compact, the continuous 
	injection \[C_\Cmp(X)\longrightarrow C_\Fin(X)\] restricts
	to a closed embedding on $\overline F$.
\end{proof}
The isometry group $\Isom(X)$ acts in a natural way from the left on $C_\fS(X)$, for $\fS=\Fin,\Cmp,\Bnd$.
This action fixes the subring $\RR\subseteq C_\fS(X)$ of constant functions pointwise, and from this we
get an induced left action on $C_\fS(X)/\RR$.
If we put
\[
(g\psi)(x)=\psi(g^{-1}(x))-\psi(g^{-1}(o)),
\]
for $g\in\Isom(X)$ and $\psi\in I_{\fS,o}(X)$, then the diagram 
\[
\begin{tikzcd}
C_{\fS}(X) \arrow{r}{\pr_o} \arrow{d} & I_{\fS,o}(X) \\
C_{\fS}(X)/\RR \arrow{ru}
\end{tikzcd}
\]
is $\Isom(X)$-equivariant.

Suppose that $A\subseteq X$ is a closed subset.
We put $\fS|A=\{Q\cap A\mid Q\in\fS\}$.
For $\fS=\Fin,\Cmp,\Bnd$, the set $\fS|A$ is the set of finite/compact/bounded subsets of $A$.
We have a natural continuous restriction homomorphism of topological vector spaces 
\[
C_\fS(X)\longrightarrow C_{\fS|A}(A)
\]
that maps a function $\phi$ to its restriction $\phi|A$.
In the commutative diagram
\[
\begin{tikzcd}
C_\fS(X)\arrow{r} \arrow{d}{\pr} \arrow{dr} & C_{\fS|A}(A) \arrow{d} \\
C_\fS(X)/\RR \arrow{r}{r_A} & C_{\fS|A}(A)/\RR,
\end{tikzcd}
\]
the map $\pr$ is open and thus $r_A$ is a continuous homomorphism.


\section{Asymmetric norms and asymmetric metrics}\label{sec:metrics-norms}

\begin{Def}\label{asmetricdef}
	Let $X$ be a set. An \emph{asymmetric metric} on $X$ is a map $\delta:X\times X\longrightarrow \RR$
	such that the following hold for all $u,v,w\in X$.
	\begin{enumerate}[(i)]
		\item $\delta(u,v)\geq 0$.
		\item $\delta(u,v)=0$ if and only if $u=v$.
		\item $\delta(u,w)\leq \delta(u,v)+\delta(v,w)$.
	\end{enumerate}
\end{Def}
In contrast to a metric we do not require that $\delta(u,v)=\delta(v,u)$.
Thus, every metric is in particular an asymmetric metric.
If $\delta$ is an asymmetric metric on $X$, we put
\[
\mathrm{Isom}_\delta(X)=\{g\in\mathrm{Isom}(X)\mid \delta(g(u),g(v))=\delta(u,v)\text{ for all }u,v\in X\}.
\]
In a similar vein, we may define asymmetric norms. 
\begin{Def}\label{asnormdef}
	An \emph{asymmetric norm} on a real vector space $V$  is a map $\nu:V\longrightarrow\RR$ such that
	the following hold for all $u,v\in V$ and all $r\geq0$.
	\begin{enumerate}[(i)]
		\item $\nu(u)\geq0$.
		\item $\nu(u)=0$ if and only if $u=0$.
		\item $\nu(ru)=r\nu(u)$.
		\item $\nu(u+v)\leq \nu(u)+\nu(v)$.
	\end{enumerate}
\end{Def}
In contrast to a norm, we do not require that $\nu(u)=\nu(-u)$. Thus, every norm is also an asymmetric norm. If $\nu$ is an asymmetric norm, then 
$\overline{\nu}(u):=\max\{\nu(u),\nu(-u)\}$ is a (symmetric) norm.
\begin{Rem}	
	Any asymmetric norm $\nu$ induces an asymmetric metric $\delta$ via $\delta(u,v)=\nu(u-v)$.
	Indeed, we have for $u,v,w\in V$ that
	\[
	\delta(u,w)=\nu(u-w)=\nu((u-v)+(v-w))\leq \nu(u-v)+\nu(v-w)=\delta(u,v)+\delta(v,w).
	\]
\end{Rem}
\begin{Lem}\label{LipschitzLemma}
	Let $(V,||.||)$ be a normed real vector space (not necessarily finite dimensional) and assume that $B\subseteq V$ is a closed convex, bounded neighborhood of
	$0$. Put 
	\[
	\nu(u)=\inf\{\lambda\geq 0\mid u\in \lambda B\}.
	\]
	Then $\nu$ is an asymmetric norm, with unit ball $B$. 
	Moreover, there exist real constants $\alpha,\beta>0$  such that \[
	||u||\leq\alpha\nu(u)
	\text{ and }
	\nu(u)\leq\beta||u||
	\] hold for all $u\in V$.
\end{Lem}
\begin{proof} 
	It is clear from the definition that Condition (i) from Definition~\ref{asnormdef} holds.
	Since $B$ is bounded, there exists $\alpha>0$ such that the closed ball $\bar B_\alpha(0)$ of $||.||$-radius $\alpha$
	around $0$ contains $B$.
	Since \[r(1+\eps)B\subseteq r(1+\eps)\bar B_\alpha(0)=\bar B_{r(1+\eps)\alpha}(0)\] holds for all $r\geq0$ and all $\eps>0$,
	we have $||u||\leq \alpha r$ if $\nu(u)\leq r$. This shows (ii), and also that $||u||\leq\alpha\nu(u)$.
	If $r>0$, then $ru\in\lambda B$ if and only if $u \in \frac{\lambda}{r} B$.  This shows (iii).  
	For $s,t>0$ the convexity of $B$ implies that $s B+t B\subseteq(s+t)B$.
	Suppose that $u,v\in V$ with $s=\nu(u)$ and $t=\nu(v)$.
	For all $\eps >0$ we have then $u\in(\eps+s)B$, $v\in(\eps+t)B$, whence
	$u+v\in(s+t+2\eps )B$. Therefore $\nu(u+v)\leq \nu(u)+\nu(v)+2\eps$. Since this
	holds for all $\eps>0$, we have (iv). 
	Since $B$ is a $0$-neighborhood, there exists $\beta>0$ such that 
	$\bar B_\frac{1}{\beta}(0)\subseteq B$.
	Then \[\bar B_{r(1+\eps)}(0)\subseteq \beta r(1+\eps)B\] for all $r\geq 0$ and $\eps>0$, and thus
	$||u||\leq r$ implies that $\nu(u)\leq \beta||u||$.
\end{proof}
In the converse direction we have the following.
\begin{Lem}\label{FiniteDimAsm}
	Let $\nu$ be an asymmetric norm on a finite dimensional real vector space $V$. Then there
	is a unique compact convex $0$-neighborhood $B\subseteq V$ such that 
	\[\nu(v)=\inf\{\lambda\geq 0\mid v\in \lambda B\}.\]
\end{Lem}
\begin{proof}
	We put $B=\{u\in V\mid \nu(u)\leq 1\}$ and $m=\dim(V)$.
	From the definition of $B$ we have $\nu(u)=\inf\{\lambda\geq 0\mid u\in \lambda B\}$.
	By the triangle inequality for $\nu$, the set $B$ is convex.
	Let $e_1,\ldots,e_m$ be a basis for $V$, and put $r=\max_j\{\nu(e_j),\nu(-e_j)\}$.
	Then $B$ contains the convex hull of the $2m$ points $\pm\frac{1}{r}e_j$. In particular, $B$ is a
	convex neighborhood of $0$ in the standard topology of $V$. We claim that $B$ is bounded
	with respect to the Euclidean norm $||.||$ determined by the basis $e_1,\ldots,e_m$.
	Otherwise we find a $||.||$-convergent sequence $(u_k)_{k\geq 1}$ of $||.||$-unit vectors $u_k$
	such that $ku_k\in B$ holds for all $k\geq 1$. Here we use that 
	the closed unit ball $\bar B_1(0)$ is compact, because $V$ has finite dimension.
	We put $u=\lim_ku_k$ and we note that $u\neq0$,
	since $||u||=1$. Given $s>0$ and $k$ large enough, we have $su_k\in B$
	(because $B$ is convex)
	and $s(u-u_k)\in B$ (because $B$ is a $0$-neighborhood). Hence $\frac{s}{2}u\in B$ for all $s>0$.
	But then $\nu(u)=0$, a contradiction. Hence $B$ is bounded.
	Since $B$ is a $0$-neighborhood, it contains a ball $B_\eps(0)$, for some $\eps>0$.
	Therefore $\nu$ is $\frac{1}{\eps}$-Lipschitz and in particular continuous. Hence $B$ is closed and thus
	compact.
	If $A\subseteq V$ is a compact convex identity neighborhood 
	with $\nu(v)=\inf\{\lambda\geq 0\mid v\in \lambda A\}$, then 
	$A=\{u\in V\mid\nu(u) =1\}$ and thus $A=B$.
\end{proof}


\section{The asymmetric bordification}\label{sec:asymmetric-bordification}

In this section we introduce horofunction bordifications of metric spaces with respect to asymmetric metrics. We discuss conditions under which the space $X$ can be topologically embedded into its bordification.   

Throughout the section we assume that $(X,d)$ is a metric space and that $\delta:X\times X\longrightarrow\RR$ is an asymmetric metric on $X$ which is \emph{bi-Lipschitz} equivalent 
to $d$. That is, there exist real constants $\alpha,\beta>0$ such that the following condition (bL) holds:
\begin{itemize}
	\item[(bL)]
	for all $p,q\in X$, we have $d(p,q)\leq\alpha \delta(p,q)$ and $\delta(p,q)\leq\beta d(p,q)$.
\end{itemize}
For a finite dimensional Euclidean vector space every asymmetric norm has Property (bL) by Lemma~\ref{LipschitzLemma} and Lemma~\ref{FiniteDimAsm}.

\begin{Lem}
	Assume that $(X,d)$ is a metric space and that $\delta:X\times X\longrightarrow\RR$ is an asymmetric metric satisfying (bL). Then the map
	\[
	\iota_\fS:X\longrightarrow C_\fS(X), \qquad p\longmapsto \delta_p=\delta(p,-)
	\]
	is an embedding for $\fS=\Fin,\Cmp,\Bnd$.
\end{Lem}
\begin{proof}
	First of all we notice that
	\begin{equation}
	\label{InA} 
	\delta_p(x)-\delta_p(y)\leq \delta(y,x)\leq\beta d(x,y).
	\end{equation}
	Hence each $\delta_p$ is Lipschitz continuous and therefore contained
	in $C_\Bnd(X)$.
	For all $p,q,x\in X$ we have
	\[
	\delta_p(x)-\delta_q(x)\leq \delta(p,q)\leq \beta d(p,q).
	\]
	This shows that {the maps} $\iota_\Bnd, \iota_\Cmp$ and  $\iota_\Fin$ are continuous.
	Since $p$ is the unique minimum of $\delta_p$, these maps are injective.
	
	It suffices to show that $\iota_\Fin$ is an embedding. This will imply that
	$\iota_\Bnd$ and $\iota_\Cmp$ are also embeddings.
	Suppose that $A\subseteq X$ is closed and that $p\in X-A$. Then there exists $\eps>0$
	such that $\bar B_\eps(p)\cap A=\emptyset$. For $y\in A$ we have that
	\[\delta_y(p)\geq\frac1\alpha d(y,p)\geq\frac1\alpha\eps.\] 
	Since the evaluation map 
	$\phi\longmapsto\phi(p)$ is continuous on $C_\Fin(X)$ and since $\delta_p(p)=0$,
	we see that $\iota_\Fin(p)\not\in\overline{\iota_\Fin(A)}$
	(where the closure is taken with respect to $\mathcal T_\Fin$).
	Thus $\iota_\Fin$ is an embedding.
\end{proof}

\begin{Cor}
	For every $o\in X$ the map 
	\[\iota_{\fS,o}:X\longrightarrow I_{\fS,o}(X), \qquad p\longmapsto \delta_p-\delta_p(o)\]
	is a continuous injection, with respect to $\fS=\Fin,\Cmp$ and $\Bnd$.
\end{Cor}
\begin{proof}
	The map $\delta_p-\delta_p(o)$ has a unique minimum at the point $p$. Hence $\iota_{\fS,o}$
	is injective.
	The map $\iota_{\fS,o}$ is the composite of the continuous map 
	$\iota_{\fS}$
	and of the continuous projector 
	${\pr_o:C_\fS(X)\longrightarrow I_{\fS,o}(X)}$ and hence continuous.
\end{proof}

\begin{Rem}
	Contrary to claims made in the literature (eg. \cite[p.~268]{BH}) the map $\iota_{\fS,o}$ need \emph{not} be an embedding
	with respect to $\fS=\Fin,\Cmp$ or $ \Bnd$, even if $X$ is proper.
	For an example, put $X=\mathbb N$, with the metric 
	\[
	d(k,\ell)=
	\begin{cases}
	0&\text{for }k=\ell \\
	k+\ell&\text{else}.
	\end{cases}
	\]
	Then $(\mathbb N,d)$ is a discrete proper metric space, whence
	$\Fin=\Cmp=\Bnd$.
	Put $\delta=d$. For $o=0$ we have
	\[\iota_{\Fin,o}(k)(\ell)=(\ell+k)-k=\ell=\iota_{\Fin,o}(0)(\ell)\] 
	for all $\ell\neq k$. It follows that 
	the sequence $(\iota_{\Fin,o}(k))_{k\geq 0}$ converges pointwise to $\iota_{\Fin,o}(0)$.
	In particular, the image \[\iota_{\Fin,o}(\mathbb N)\subseteq I_{\Fin,o}(X)\] is not discrete.
\end{Rem}

We need a geometric condition on $X$ that ensures that $\iota_{\fS,o}$ is an embedding.

\begin{Def}
	We say that an asymmetric metric $\delta$ on a set $X$ satisfies the \emph{interval condition}, or has \emph{Property~(ic)} if 
	the following holds:
	\begin{itemize}
		\item[(ic)] for all $p,q\in X$ and $s\in[0,\delta(p,q)]$, there is $z\in X$ such that $\delta(p,q)=\delta(p,z)+\delta(z,q)$ and $\delta(p,z)=s$.
	\end{itemize}
\end{Def}

Every asymmetric metric induced by an asymmetric norm on a vector space has Property~(ic).
Also, every geodesic metric space $(X,d)$ has this property for $\delta=d$.

\begin{Prop}
	Assume that $(X,d)$ is a metric space and that $\delta$ is an asymmetric
	metric on $X$ satisfying (bL) and (ic). 
	Then the map \[\tilde\iota_\Bnd:X\longrightarrow C_\Bnd(X)/\RR,\qquad p\longmapsto \delta_p+\RR\]
	is a topological embedding.
\end{Prop}
\begin{proof}
	Being the composite $X\xrightarrow{\ \Bnd\ } C_\Bnd(X)\longrightarrow C_\Bnd(X)/\RR$,
	the map $\tilde\iota_\Bnd$ is continuous.
	Suppose that $A\subseteq X$ is closed and that $p\in X-A$.
	We claim that $\tilde\iota_\Bnd(p)$ is not in the closure of $\tilde\iota_\Bnd(A)$.
	For this it suffices to show that $\iota_{\Bnd,p}(p)$ is not in the closure of $\iota_{\Bnd,p}(A)$
	in $I_{\Bnd,p}(X)\cong C_\Bnd(X)/\RR$. This is what we will show.
	
	There exists $\eps>0$ such that $\bar B_\eps(p)\cap A=\emptyset$.
	We claim that for every $y\in A$ there exists a point $z$ in the bounded
	set $\bar B_{2\eps}(p)$ such that 
	\[\textstyle
	|\iota_{\Bnd,p}(p)(z)-\iota_{\Bnd,p}(y)(z)|\geq 
	\min\left\{\frac1\alpha\eps,\frac2{\alpha^2\beta}\eps\right\}.\]
	This will show that $\iota_{\Bnd,p}(p)$ is not in the closure
	of $\iota_{\Bnd,p}(A)$.
	In order to prove the claim, let $y\in A$. 
	If $d(p,y)\leq 2\eps$, we put $z=y$. 
	Then \[\iota_{\Bnd,p}(y)(z)=\delta_y(z)-\delta_y(p)=-\delta_y(p)\leq0,\] and 
	\[\textstyle
	\iota_{\Bnd,p}(p)(z)=\delta_p(z)-\delta_p(p)=\delta(p,z)\geq\frac1\alpha\eps,
	\]
	whence 
	\[\textstyle
	|\iota_{\Bnd,p}(p)(z)-\iota_{\Bnd,p}(y)(z)|\geq 
	\frac1\alpha\eps
	\]
	in this case.
	If $d(p,y)>2\eps$, then $\delta_y(p)\geq \frac2\alpha \eps$.
	By (ic) we can find a point $z\in X$ with
	$\delta_y(z)+\delta_z(p)=\delta_y(p)$ and with $\delta_z(p)=\frac2\alpha\eps$.
	Then $d(p,z)\leq 2\eps$ and 
	\[\iota_{\Bnd,p}(y)(z)=\delta_y(z)-\delta_y(p)=-\delta_z(p)\leq0,\]
	while
	\[\textstyle
	\iota_{\Bnd,p}(p)(z)=\delta(p,z)-\delta(p,p)\geq\frac1\alpha d(p,z)=\frac1\alpha d(z,p)\geq \frac{1}{\alpha\beta} \delta(z,p)
	=\frac{2}{\alpha^2\beta}\eps.
	\]
	Hence
	\[\textstyle
	|\iota_{\Bnd,p}(p)(z)-\iota_{\Bnd,p}(y)(z)|\geq 
	\frac2{\alpha^2\beta}\eps
	\]
	in this case.
\end{proof}

\begin{Prop}\label{Compactness}
	Assume that $(X,d)$ is a metric space and that $\delta$ is an asymmetric
	metric on $X$ satisfying (bL). Then $\iota_{\Cmp,o}(X)$
	has compact closure in $I_{\Cmp,o}(X)$, and 
	the same set is also the closure of $\iota_{\Fin,o}(X)$ in $I_{\Fin,o}(X)$.
\end{Prop}
\begin{proof}
	We put
	\[
	F=\{\delta_p-\delta_p(o)\mid p\in X\}\subseteq C_\Cmp(X).
	\]
	We have
	\[
	|(\delta_p(x)-\delta_p(o))-(\delta_p(y)-\delta_p(o))|\leq \delta(y,x)\leq\beta d(x,y)
	\]
	by Inequality~(\ref{InA}), which shows that $F$ is equicontinuous.
	For $x\in X$ fixed we have
	\[
	|\delta_p(x)-\delta_p(o)|\leq \delta(o,x)\leq\beta d(o,x),
	\]
	which is a bounded set. Hence we may apply Ascoli's Theorem~\ref{Ascoli}.
\end{proof}

\begin{Def}
\label{def::bordification}
	Let $(X,d)$ be a metric space and assume that $\delta$ is an asymmetric metric on $X$ 
	having properties (bL) and (ic). We call the closure of
	$\tilde\iota_\Bnd(X)$ in $C_\Bnd(X)/\RR$ the \emph{bordification} $\widehat X$ of $X$ (with respect
	to $\delta$),
	\[
	\widehat X=\overline{\tilde\iota_\Bnd(X)}\subseteq C_\Bnd(X)/\RR.
	\]
\end{Def}

If $(X,d)$ is a complete CAT(0) space, then this construction gives, for $\delta=d$, the 
bordification of $\widehat X=X\cup\partial_\infty X$ by its visual boundary $\partial_\infty X$, as described for example in \cite[III.8]{BH}.

The space $\widehat X$ is a complete uniform space.
The uniform structure on $C_\Bnd(X)/\RR$ is defined by means of the countable family $(d_k)_{k\in\NN}$ 
of pseudo-metrics
\[
d_k(\phi,\psi)=\sup\{|(\phi(x)-\phi(o))-(\psi(x)-\psi(o))| \mid x\in X\text{ and }d(x,o)\leq 2^k\}, 
\]
where $o\in X$ is a fixed basepoint.

If $(X,d)$ as in Definition \ref{def::bordification} is in addition proper, then $\mathcal T_\Bnd=\mathcal T_\Cmp$ and thus $\widehat X$ is compact by Proposition~\ref{Compactness}.
The representatives in $C_\Bnd(X)$ of the elements of $\widehat X$ are called \emph{horofunctions}. 
\emph{Horoballs} are the sublevel sets of horofunctions.
Every horofunction $h$ has a unique representative in $I_{\Bnd,o}(X)$, namely $h-h(o)$.  
We call these representatives \emph{normalized horofunctions} with respect to the base point $o$.
The group $\Isom_\delta(X)$ acts in a natural way from the left on the bordification of $X$,
because every isometry $g$ preserves $\delta$, and maps $N_{Q,\eps}(\delta(p,-))$ to $N_{g(Q),\eps}(\delta(gp,-))$. 
We recall that the action of $\Isom_\delta(X)$ on the set of normalized horofunctions
is given by
\[
(g\phi)(x)=\phi(g^{-1}(x))-\phi(g^{-1}(o)).
\]


\section{Horofunctions via nonstandard analysis}
\label{sec:5}

Our aim now is to describe horofunctions using nonstandard analysis. 
We fix a free ultrafilter $\mu$ on a countably infinite index set $J$.
Given any set $X$, we denote by $\ns X$ the ultrapower of $X$ with respect to $\mu$. Thus
\[
\ns X=\prod_JX/\mu,
\]
where two sequences $x,y\in \prod_JX$ are identified in $\prod_JX/\mu$ if $\{j\in J\mid x_j=y_j\}\in\mu$.
There is a natural diagonal injection \[X\longrightarrow \ns X\]
which allows us to view $X$ as a subset of $\ns X$.
If $f:X\longrightarrow Y$ is a function, then $f$ has a natural extension $\ns f:\ns X\longrightarrow\ns Y$.
If $X=\RR$, then $\ns\RR$, endowed with the extended multiplication and addition, is a field, the field 
of \emph{nonstandard reals}. \L os' Theorem \cite[5.2.1]{BellSlomson} guarantees that the ultrapower of a given 
first-order structure satisfies exactly the same first-order formulas as the original first-order structure.
Thus $\ns \RR$ is an ordered real closed field, because this is a first-order property: we may write out a 
sentence, for each $n\geq 1$, saying that every polynomial of degree $2n+1$ has a zero.
Likewise, we can write out that every positive element is a square. We put $|r|=\max\{\pm r\}$, for $r\in\ns\RR$.

The reason why nonstandard structures are interesting is that they contain in general new elements
with remarkable properties. This phenomenon is called \emph{$\omega_1$-saturation} of ultrapowers. 
If $f_n$ is a countable sequence of first-order formulas in a free variable 
and if for each $n$ there is an element $x_n\in X$ that witnesses $f_k$ for all $k\leq n$, then there is an element 
$x\in\ns X$ that witnesses all formulas $f_n$ simultaneously, see eg.~\cite[11.2.1]{BellSlomson}.
For instance, there exists for every $n\in\NN$ a real number $r$ such that $r>k$, for $k=0,1,2,\ldots,n$ (eg. $r=n+1$).
It follows that in $\ns\RR$, there exist elements $r$ such that $r>n$ holds for every natural number $n$,
i.e. $\ns \RR$ is a \emph{non-archimedean ordered real closed field}
\footnote{An ordered field is called \emph{archimedean} if for every field element $r$, there exists an integer
	$n$ such that $r\leq n$.}
which contains $\RR$ as a subfield.
The set of \emph{finite elements} in $\ns\RR$ is defined as 
\[
\ns\RR_{\fin}:=\{r\in\ns\RR\mid |r|\leq n\text{ for some }n\in\NN\}.
\]
This subset is a local ring, whose unique maximal ideal is the set of \emph{infinitesimal elements}, defined as 
\[
\ns\RR_{\infi}:=\{r\in\ns\RR\mid |r|\leq 2^{-n}\text{ for every }n\in\NN\}.
\]
The natural map $\std:\ns\RR_{\fin}\longrightarrow\ns\RR_{\fin}/\ns\RR_{\infi}\cong\RR$
is called the \emph{standard part map}. It splits surjectively as
\[
\begin{tikzcd}
0 \arrow{r} & \ns\RR_{\infi} \arrow{r}{\unlhd} & \ns\RR_{\fin}\arrow{r}{\std} & \RR \arrow{r} \arrow[shift left]{l}{\inc}  & 0,
\end{tikzcd}
\]
compare~\cite{Robinson}~9.4.3.

\begin{Def}
	Assume that $(X,d)$ is a metric space with basepoint $o$ and that $\delta$ is an asymmetric metric on $X$ which satisfies conditions (bL) and (ic). 
	Then $d$ and $\delta$ extend to maps 
	\[
	\ns d:\ns (X\times X)=\ns X\times\ns X\longrightarrow\ns\RR\quad \text{ and }\quad
	\ns \delta:\ns (X\times X)=\ns X\times\ns X\longrightarrow\ns\RR.
	\]
	By \L os' Theorem, $\ns \delta$ and $\ns d$ have the same first-order properties as $\delta$ and $d$.
	In particular, they satisfy the axioms (i)-(iii) from Definition~\ref{asmetricdef}, and the 
	conditions (bL) and (ic).
	
	We recall that $X$ may be viewed as a subset of $\ns X$.
	For $p\in\ns X$ we define a map $h_p:X\longrightarrow\RR$ by
	\[
	h_p(x):=\std(\ns \delta(p,x)-\ns \delta(p,o)).
	\]
\end{Def}

The right-hand side is well-defined, since the triangle inequality for $\ns \delta$ implies that $\ns \delta(p,x)-\ns \delta(p,o)\in\ns\RR_{\fin}$ for all $x\in X$. 
We note that $h_p$ is $\beta$-Lipschitz for the constant $\beta$ in Condition (bL) and that $ h_p\in I_{\Bnd,o}(X)$.
For $p\in X$ we obtain $h_p=\delta_p-\delta_p(o)$, which is a horofunction. 

In general, not every such $h_p$ is a horofunction. We therefore introduce the following notion. 
Let us call a metric space $(X,d)$ \emph{almost proper} if the following holds.
For every bounded set $Y\subseteq X$ and every $\eps>0$, there exists a finite set 
$Y_0\subseteq X$ such that $Y\subseteq\bigcup\{ B_\eps(x)\mid x\in Y_0\}$, where $B_\eps(x):=\{ y \in X\mid d(x,y) <\eps \}$. 
\footnote{In other words, we require that every bounded subset of $X$ is totally bounded.}
Every proper metric space is almost proper. Conversely, the metric completion
of an almost proper metric space is proper. An example of a non-proper, almost proper metric space is $\QQ$, endowed with the Euclidean metric induced from $\RR$.

\begin{Thm}\label{Thm41}\label{thm:normalized-horofunction}
	Let $(X,d)$ be a metric space and assume that $\delta$ is an asymmetric metric on $X$ 
	having properties (bL) and (ic). Let $o\in X$ be a basepoint.
	For every normalized horofunction $\phi\in\widehat X$, there exists $p\in\ns X$ with $\phi=h_p$.
	If $(X,d)$ is almost proper, then conversely every $h_p$, for $p\in\ns X$, is a normalized horofunction.
\end{Thm}
\begin{proof}
	Suppose that $\phi:X\longrightarrow \RR$ is a normalized horofunction. 
	Let $\ns\phi:\ns X\longrightarrow\ns\RR$ denote its extension to the
	ultrapower and consider the 
	countable set $F=\{f_{k,\ell}(\mathsf{v})\mid k,\ell\in\NN\}$ of formulas $f_{k,\ell}$ in one free variable $\mathsf v$,
	\[
	f_{k,\ell}(\mathsf{v})=\mathsf{\forall x \left[d(x,o)\leq 2^{\mathit k} \rightarrow |\phi(x)-(\delta(v,x)-\delta(v,o))|\leq 2^{-\ell}\right]}.
	\]

	For every finite subset $F_0\subseteq F$, there exists a point $p\in X$ such that
	if we substitute $p$ for the free variable $\mathsf{v}$, then
	$f_{k,\ell}(p)$ holds simultaneously for all formulas $f_{k,\ell}\in F_0$.
	This is true since $\phi$ is a horofunction,
	which can be approximated to arbitrary precision on each ball of radius
	$2^k$ by a map $x\longmapsto \delta(p,x)-\delta(p,o)$, for some choice of $p\in X$.
	By the aforementioned $\omega_1$-saturation of ultraproducts based on countable index sets,
	there exists a point $p\in \ns X$ such that 
	$f_{k,\ell}(p)$ holds simultaneously for all formulas $f_{k,\ell}\in F$.
	Hence
	\[
	\phi(x)-(\ns \delta_p(x)-\ns \delta_p(o))\in\ns\RR_{\infi}
	\]
	holds for all $x\in X$, that is, $h_p=\phi$.
	
	Assume now that $(X,d)$ is proper, and that $\delta\leq\beta d$. 
	Since $X$ is almost proper, we find finite sets $Y_{k,\ell}\subseteq X$ such 
	that 
	$\bar B_{2^k}(0)\subseteq\bigcup\{\bar B_{2^{-\ell}}(y)\mid y\in Y_{k,\ell}\}$.
	Let $p\in\ns X$ and put
	$h_p(x)=\std(\ns\delta_p(x)-\ns\delta_p(o))$.
		Let $(p_j)_{j\in J}$ be a
	sequence in $\prod_J X$ representing $p$ in the 
	ultrapower $\prod_JX/\mu$.
	We put 
	\[
	\phi_j(x)=\delta(p_j,x)-\delta(p_j,o)
	\]
	and we note that these maps are normalized horofunctions.
	Given $\ell\in\NN$ and $x\in X$, the set \[J_\ell(x) := \{ j \in J \mid |h_p(x)- \phi_j(x)| \leq 2^{-\ell}\beta\}\] is in 
	the ultrafilter $\mu$, by the definition of $h_p$.
	Since $Y_{k,\ell}$ is finite, the set \[J_{k,\ell}=\bigcap\{ J_\ell(y)\mid y\in Y_{k,\ell}\}\] is also in $\mu$
	and in particular nonempty.
	For $x\in \bar B_{2^k}(0)$, there exists $y\in Y_{k,\ell}$ with $d(x,y)\leq2^{-\ell}$.
	For $j\in J_{k,\ell}$ we have thus
	\begin{align*}
	|h_p(x)-\phi_j(x)| 
	& \leq|h_p(x)-h_p(y)|+|h_p(y)-\phi_j(y)|+|\phi_j(y)-\phi_j(x)|\\
	& \leq 3\cdot 2^{-\ell}\beta.
	\end{align*}
	Thus the set of normalized horofunctions 
	$\{\phi_j\mid j\in J\}$ has $h_p$ in its closure (with respect to the topology $\mathcal T_\Bnd$),
	whence $h_p\in\widehat X$.
\end{proof}
If $X$ is not almost proper, not every $h_p$ needs to be a horofunction.
Put $X=\NN$ with the discrete metric $d(i,j)=1$
if $i\neq j$. Then $\widehat X=X$. Put $J=\NN$, and consider the point
$p\in\ns X$ which is represented by the sequence $(0,1,2,3,\ldots)\in\NN^\NN$.
The $\ns d(k,p)=1$ for every $k\in X$, whence $h_p=0$.
\begin{Cor}\label{cor:restrictions}
	Let $(X,d)$ be an almost proper metric space and assume that $\delta$ is an asymmetric metric on $X$.
	Let $A\subseteq X$ be a closed subset and assume that both $\delta$ and the restriction of $\delta$ to $A$
	have properties (bL) and (ic).
	Then every horofunction on $A$ is the restriction of some horofunction on $X$.
\end{Cor}
\begin{proof}
	Let $\phi$ be a horofunction on $A$. We may assume that $\phi$ is normalized with respect to a base point $o\in A$.
	There exists $p\in\ns A$ such that  $\phi(x)=\std(\ns \delta(p,x)-\ns \delta(p,o))$. Since $X$ is almost proper
	and since $\ns A\subseteq\ns X$, the map $h_p$ is a horofunction on $X$, with $h_p|A=\phi$.
\end{proof}
In the setting of proper metric spaces the previous corollary follows also directly.
If $A\subseteq X$ is a subspace, we may consider the commutative diagram
\[
 \begin{tikzcd}
  A \arrow{r}{\inc}\arrow{d}{\tilde\iota_{\Cmp|A}} & X \arrow{d}{\tilde\iota_{\Cmp}} \\
  C_{\Cmp|A}(A)/\RR & C_\Cmp(X)/\RR \arrow{l}{r_A}.
 \end{tikzcd}
\]
By continuity we have an inclusion $r_A(\overline{\tilde\iota_\Cmp(\inc(A))})\subseteq\overline{\tilde\iota_{\Cmp|A}(A)}$
and by compactness of the set $\overline{\tilde\iota_{\Cmp}(\inc(A))}$ we have equality.
Hence $r_A(\overline{\tilde\iota_{\Cmp}(\inc(A))})=\widehat A$ if $A$ and $X$ are proper and satisfy (bL) and (ic).
In particular, $r_A$ maps $\overline{\tilde\iota_{\Cmp}(\inc(A))}$ homeomorphically onto $\widehat A$ if and only if
$r_A$ is injective on $\overline{\tilde\iota_{\Cmp}(\inc(A))}$.
\color{black}
%
%


\section{Polyhedral norms}
\label{sec:polyhedral-norms}

In this section we introduce asymmetric norms determined by compact convex polyhedra. 
We fix a finite dimensional real vector space $V$, with dual $V^\vee=\mathrm{Hom}_\RR(V,\RR)$ and assume that $d(u,v)=||u-v||$
is a Euclidean metric on $V$. Then the metric space $(V,d)$ is proper.
We also fix $o=0\in V$ as the base point.

Let $B\subseteq V$ be a compact convex polyhedral $0$-neighborhood and let $A_0,\ldots,A_m\subseteq B$ be the codimension-$1$-faces
of $B$. Corresponding to each $A_j\subseteq B$, there is a unique linear functional $\xi_j\in V^\vee$ such that $A_j=\{v\in B\mid\xi_j(v)=1\}$.
This allows us to write $B$ as 
\[
B=\{u\in V\mid \xi_0(u),\cdots,\xi_m(u)\leq 1\}.
\]
The asymmetric norm $\nu$ determined by $B$ as in Lemma~\ref{LipschitzLemma} is then given by
\[
\nu(u)=\max\{\xi_0(u),\ldots,\xi_m(u)\}.
\]
We put $K=\{0,\ldots,m\}$.
A nonempty subset $L\subseteq K$ is called a \emph{dual face} if there exists $v\in V$ with $\nu(v)=1$ such that \[L=\{k\in K\mid\xi_k(v)=1\}.\] 
The geometric motivation for this is as follows.
The set $B$ has a polyhedral dual $B^\vee\subseteq V^\vee$, which is given by
\[
B^\vee=\{\xi\in V^\vee\mid\xi(u)\leq 1\text{ for all }u\in B\}.
\]
Thus $B^\vee$ is the convex hull of $\xi_0,\ldots,\xi_m$.
The proper faces of the polyhedron $B^\vee$ are precisely the convex hulls
of the sets $\{\xi_\ell\mid\ell\in L\}$, where $L\subseteq K$ is a dual face as defined above. 
We emphasize that a dual face in our setup is just a subset of the index set $K$.  
We denote the set of all dual faces by
\[
\Sigma=\{L\subseteq K\mid L\text{ is a dual face}\}.
\]
For any nonempty subset $L\subseteq K$ we put 
\[
\nu_L(u)=\max\{\xi_\ell(u)\mid\ell\in L\}.
\]
Thus $\nu=\nu_K$. The \emph{negative cone} of $L$ is the set 
\[
N_L=\{v\in V\mid\xi_\ell(v)\leq0\text{ for all }\ell\in L\}.
\]
If $W\subseteq V$ is a linear subspace with $W\cap N_L=\{0\}$, then the restriction 
$\nu_L|W$ is an asymmetric norm on $W$.

\begin{Lem}\label{TriangleInEq}
	There is a real constant $\gamma>0$ such that 
	\[
	|\nu_L(p)-\nu_L(q)|\;\leq\;\gamma||p-q||
	\]
	and
	\[
	|\nu_L(p-u)-\nu_L(p)-\nu_L(q-u)+\nu_L(q)|\;\leq\;2\gamma||p-q|| 
	\]
	hold for all subsets $L\subseteq K$ and all $u,p,q\in V$.
\end{Lem}
\begin{proof}
	We choose $\gamma$ in such a way that $|(\xi_k-\xi_\ell)(u)|\leq\gamma||u||$ holds for all $k,\ell\in K$
	and $u\in V$.
\end{proof}

The metric space $(V, d)$ is proper. If we put \[\delta(u,v)=\nu(u-v),\]
the asymmetric metric $\delta$ satisfies conditions (bL) and (ic). 
Its normalized horofunctions are by Theorem~\ref{Thm41} the maps 
\[
h_p(v)=\std(\ns\nu(p-v)-\ns\nu(p)),
\]
for $p\in\ns V$. The first aim of this section is to show the following.

\begin{Thm}\label{AsymHoro}
	The normalized horofunctions of $V$ with respect to the asymmetric metric $\delta$ as above are precisely the maps 
	\[
	u\longmapsto\nu(p-u)-\nu(p),
	\]
	for $p\in V$, and the maps 
	\[
	u\longmapsto\nu_L(p-u)-\nu_L(p),
	\]
	for $p\in V$ and $L\subseteq K$ a dual face.
\end{Thm}

The following picture illustrates the result. It shows on the left the polygonal unit sphere $h_o=1$, and on the right 
the level set $h_p=0$, for $p$ at infinite distance from $o$.

\begin{center}
    \colorlet{shadecolor}{gray!40}
	\begin{tikzpicture}[scale=0.5]
		\draw[shading=radial,inner color=shadecolor,style=dashed] (0,0) circle(2);
		\node at (3,-1) {$X$};
		\node at (0.5,0.5) {$o$};
		\path[draw] (-.5,-.5) -- (.7,-.5) -- (-.5,.7) --cycle;
		\draw[fill=black] (0,0) circle(0.05);
		\node at (0,-1.1) {$h_o=1$};
	\end{tikzpicture}
	\hspace{3cm}
	\begin{tikzpicture}[scale=0.5]
		\draw[shading=radial,inner color=shadecolor,style=dashed] (0,0) circle(2);
		\node at (3,-1) {$X$};
		\path[draw] (-2.2,-.3) -- (.3,-.3) -- (-1.5,1.5);
		\node at (-3.2,1.6) {$p$};
		\draw[fill=black] (-2.5,1.6) circle(0.05);
		\node at (0.5,0.5) {$o$};
		\draw[fill=black] (0,0) circle(0.05);
		\node at (-4,-.3) {$h_p=0$};
	\end{tikzpicture}
	
\end{center}

The proof of this theorem requires some preparations and can be found on page \pageref{proof}. 

\begin{Lem}\label{LittleLemma}
	Let $V$ be a finite dimensional real vector space and let $\eta_1,\ldots,\eta_n$ be nonzero linear functionals on $V$. 
	Let $d$ be a Euclidean metric on $V$.
	Then there exists a real constant $c >0 $, depending only on $\eta_1,\ldots,\eta_n$ and $d$,
	such that the following holds. If $v\in V$ is a vector with $|\eta_i(v)|\leq 1$ 
	for all $i=1,\ldots,n$, then there exists a vector $w\in\eta_1^\perp\cap\cdots\cap\eta_n^\perp$ with $d(v,w)\leq c$.
\end{Lem}
\begin{proof} 
	Recall that $\eta_i^\perp= \{ v \in V \; \vert \; \eta_i(v)=0\}$.
	First suppose $\eta_1^\perp\cap\cdots\cap\eta_n^\perp=\{0\}$.
	Then the $\eta_i$ generate the dual space $V^\vee$.
	We may assume that $\eta_1,\ldots,\eta_k$ is a basis
	for the dual space. Let $e_1,\ldots,e_k$ be the dual basis in $V$ associated to $\eta_1,\ldots,\eta_k$.
	If $v\in V$ is a vector with $|\eta_i(v)|\leq 1$ for all $i=1,\ldots,k$,
	then $v\in Q:=\{\sum_{i=1}^k e_i\lambda_i\mid \lambda_i\in[-1,1]\}\cong[-1,1]^k$. This set $Q\subseteq V$ is compact and hence bounded. 
	
	For the general case we put $H:=\eta_1^\perp\cap\cdots\cap\eta_n^\perp$ and we choose a complementary subspace $W\subseteq V$, such that $V=W\oplus H$. 
	The previous argument shows that every vector $v\in V$, 
	with $|\eta_k(v)|\leq 1$ for all $k=1,\ldots,n$,  is contained in $Q+H$, where $Q$ is compact.
	The claim follows.
\end{proof}

Our proof of Theorem~\ref{AsymHoro} will rely on the results about ultrapowers in the previous section.

The ultrapower $\ns V$ of $V$ is a finite dimensional vector space over $\ns \RR$, with dual space $(\ns V)^\vee\cong\ns (V^\vee)$.
We put \[\ns V_\fin=\{v\in\ns V\mid\ns||v||\in\ns\RR_\fin\}\] and
\[\ns V_\infi=\{v\in\ns V\mid\ns||v||\in\ns\RR_\infi\},\]
where $||.||$ is a Euclidean norm on $V$.
There is a split short exact sequence of $\ns \RR_\fin$-modules
\[
\begin{tikzcd}
0 \arrow{r} & \ns V_{\infi} \arrow{r} & \ns V_{\fin}\arrow{r}{\std} & V \arrow{r} \arrow[shift left]{l}{\inc}  & 0.
\end{tikzcd}
\]
By property (bL) we have
\begin{equation}
\label{equ::V_fin_inf}
\ns V_\fin=\{v\in\ns V\mid\ns\nu(v)\in\ns\RR_\fin\}\quad\text{ and }\quad
\ns V_\infi=\{v\in\ns V\mid\ns\nu(v)\in\ns\RR_\infi\}.
\end{equation}
By Theorem~\ref{Thm41}, the horofunctions are the maps
\[
h_p(u)=\std(\ns \delta(p,u)-\ns \delta(p,o))=\std(\ns\nu(p-u)-\ns \nu(p)),
\]
for $p\in\ns V$ and $u\in V$.
Our goal is now to analyze these horofunctions more closely. 
We define some more combinatorial data.

\begin{Def}
	For $k,\ell\in K$ we put \[H_{k,\ell}:=(\xi_k-\xi_\ell)^\perp =\{v\in V\mid\xi_k(v)=\xi_\ell(v)\} \subseteq V.\]
	For a nonempty subset $L\subseteq K$ we put
	\[
	H_L:=\bigcap_{k,\ell\in L} H_{k,\ell}=\{u\in V\mid \xi_k(u)=\xi_\ell(u)\text{ for all }k,\ell\in L\}.
	\]
\end{Def}

If $k\neq\ell$, then the set $H_{k,\ell}$ is a hyperplane in $V$. If $L=\{k\}$ then $H_L=V$ and if $L=K$ then $H_L= \{0\}$.
If $L\subseteq K$ is a dual face, then $H_L$ is the linear subspace of $V$ which intersects the affine span $F_L$
of $\{\xi_\ell\mid \ell\in L\}$ orthogonally (if we identify $V^\vee$ with $V$ via the Euclidean inner product).
In the case where $L$ is a dual face, we have thus 
\[
\dim(V)=\dim(H_L)+\dim(F_L).
\]
All these objects $\xi_i,H_{k,\ell}$ etc.~extend in a natural way as $\ns\xi_i,\ns H_{k,\ell}$ etc.~to the ultrapower $\ns V$ of $V$
which we consider now.

\begin{Lem}\label{Crucial}
	Let $L\subseteq K$ be a nonempty set. Then
	\[
	\ns H_L+\ns V_\fin=\bigcap_{k,\ell\in L}(\ns H_{k,\ell}+\ns V_\fin).
	\]
\end{Lem}
\begin{proof}
	The claim is true if $L$ consists of a single element, so we may assume
	that $L$ contains at least two elements.
	Also, the left-hand side is contained in the right-hand side, since
	\[\ns H_L+\ns V_\fin\subseteq \ns H_{k,\ell}+\ns V_\fin\]
	for $k,\ell\in L$.
	
	Let $v \in \bigcap_{k,\ell\in L}(\ns H_{k,\ell}+\ns V_\fin)$. 
	We claim that $v\in \ns H_L+\ns V_\fin$.
	Let $c> 0$ be the real constant from Lemma~\ref{LittleLemma}, for the 
	set of linear forms $\{\xi_k-\xi_\ell\mid k,\ell\in L\text{ and }k>\ell\}$.
	Since $v\in \ns H_{k,\ell}+\ns V_\fin$ holds for $k,\ell\in L$, we have that 
	$(\ns\xi_k-\ns\xi_\ell)(v) \in \ns\RR_\fin$.   
	Therefore there exists an integer $n>0$ such that for every $k,\ell\in L$ 
	\[ \ns |(\ns\xi_k-\ns\xi_\ell)(v)|  \leq n.\]
	Then $\ns |(\ns\xi_k-\ns\xi_\ell)(\frac{1}{n}v)|  \leq 1$.
	By \L os' Theorem, there exists $w\in \ns H_L$ such that $\ns||w-\frac{1}{n}v||\leq c$.
	Thus $nw-v\in \ns V_\fin$.
\end{proof}

\begin{Def}
	We say that two nonstandard reals $s,t\in\ns\RR$ have the \emph{same order of magnitude}, denoted by $s\approx t$ if 
	$s-t\in\ns\RR_{\fin}$.
	Since $\ns\RR_\fin\subseteq\ns \RR$ is a subgroup, this is an equivalence relation.
	For $p\in\ns V$ we put 
	\[
	K_p=\{k\in K\mid \ns\xi_k(p)\approx \ns\nu(p)\}
	\]
	and we note that 
	\begin{equation}
	\label{eq2}
	\ns\nu(p)=\max\{\ns\xi_k(p)\mid k\in K\}=\max\{\ns\xi_k(p)\mid k\in K_p\}=\ns\nu_{K_p}(p).
	\end{equation}
	If $p\in\ns V_\fin$, then $K_p=K$.
\end{Def}

\begin{Lem}
	\label{lem::56}
	For $p\in \ns V$ and $q\in p+\ns V_\fin$ we have \[K_q=K_p.\]
\end{Lem}
\begin{proof}
	Since $p-q \in \ns V_\fin$ we have $\ns\nu(p-q),\ns\nu(q-p)\in\ns\RR_\fin$ by Equation~(\ref{equ::V_fin_inf}). 
	Now $ \ns\nu(p)  \leq \ns\nu(p-q) + \ns \nu(q)$ and $ \ns\nu(q) \leq \ns\nu(q-p) + \ns \nu(p)$, whence $\ns\nu(p)\approx\ns\nu(q)$.
	For all $k\in K$ we have $\ns\xi_k(p-q)\in\ns\RR_\fin$, whence $\ns \xi_k(p)\approx\ns\xi_k(q)$.
	The claim follows.
\end{proof}

We record at this stage the following.

\begin{Lem}\label{intermediate}
	The normalized horofunctions of $V$ are the maps
	\[
	h_p(u)=\std(\ns\nu_{K_p}(p-u)-\ns\nu_{K_p}(p)),
	\]
	for $p\in\ns V$.
\end{Lem}
\begin{proof}
	This is true for $u\in V$, since $K_{p-u}=K_p$ by Lemma~\ref{lem::56}, and since $\ns\nu(p-u)=\ns\nu_{K_{p-u}}(p-u)$
	by Equation (\ref{eq2}).
\end{proof}

\begin{Def}
	We write $s\gg 0$ if $s\in\ns\RR$ is a nonstandard real with $s>n$ for all $n\in\NN$
	(an infinitely large nonstandard real), and we write $s\gg t$ if $s-t\gg 0$.
	For a subset $L\subseteq K$ we put 
	\[
	\ns H_L^\larg: =\{v\in\ns H_L\mid \ns\xi_\ell(v)\gg\ns\xi_k(v)\text{ for all }\ell\in L\text{ and all }k\in K\setminus L\}.
	\]
	If $q\in\ns H_L^\larg$, then $K_q=L$. Note that $\ns H_K^\larg=\{0\}$.
\end{Def}

\begin{Lem}\label{CarryOn}
	For $p\in \ns V$ we have 
	\[
	p\in\ns H_{K_p}^\larg+\ns V_\fin.
	\]
	In particular, $\ns H_{K_p}^\larg\neq\emptyset$.
\end{Lem}
\begin{proof}
	If $K_p=\{k\}$, then $\ns H_{\{k\}}=\ns V$. Moreover, $\ns\xi_k(p)\gg\ns\xi_\ell(p)$ for all $\ell\neq k$ and thus 
	$p\in\ns H_{\{k\}}^\larg$.
	If $k,\ell\in K_p$ are different indices, then $\ns\xi_k(p)\approx\ns\xi_\ell(p)$ and thus $p\in\ns H_{k,\ell}+\ns V_\fin$. 
	By Lemma~\ref{Crucial} we have
	$p=p_1+p_2$, with $p_1\in \ns H_{K_p}$ and $p_2\in\ns V_\fin$. Suppose that $k\in K_p$ and $\ell\in K\setminus K_p$.
	Then $\ns\xi_k(p_1)\approx\ns\xi_k(p)\gg\ns\xi_\ell(p)\approx\ns\xi_\ell(p_1)$. Thus $p_1\in\ns H_{K_p}^\larg$
	and $p\in\ns H_{K_p}^\larg+\ns V_\fin$.
\end{proof}

Note that $p\in\ns V_\fin\text{ if and only if }K_p=K.$
For the remaining points $p\in\ns V$ we have the following result.

\begin{Lem}
	\label{lem::dual_face_K_p}
	For every $p\in \ns V\setminus\ns V_\fin$, the set $K_p$ is a dual face.
\end{Lem}
\begin{proof}
	We put $p=p_1+p_2$, with $p_1\in\ns H_{K_p}^\larg$ and $p_2\in\ns V_\fin$, as in Lemma~\ref{CarryOn}. 
	Then $K_p=K_{p_1}$ by Lemma~\ref{lem::56}.
	Also, $p_1\neq 0$ because $p\not\in \ns V_\fin$. We put $q=\frac{1}{\ns\nu(p_1)}p_1$. Then $\ns\nu(q)=1$
	and \[\textstyle K_p=K_{p_1}=\{k\in K\mid\ns \xi_k(q)=1\},\]
	because $\ns\xi_k(p_1)\gg\ns\xi_\ell(p_1)$ holds for all $\ell\in K\setminus K_{p_1}$ and all $k\in K_{p_1}$.
	\L os' Theorem shows that a subset $L\subseteq K$ is a dual face if and only if 
	there exists $v\in\ns B$ with $\ns\nu(v)=1$, such that 
	$L=\{\ell\in K\mid\ns\xi_\ell(v)=1\}$. Hence $K_p$ is a dual face.
\end{proof}

We are now ready to prove Theorem~\ref{AsymHoro}. 

\begin{proof}[Proof of Theorem~\ref{AsymHoro}.]\label{proof}
By Lemma \ref{intermediate}, a normalized horofunction for $V$ with respect to the asymmetric metric $\delta$ is of the form 
$h_p(u)=\std(\ns\nu_{K_p}(p-u)-\ns\nu_{K_p}(p))$, for some $q \in \ns V$. 
	Let $u\in V$.
	If $q\in\ns V_\fin$,  we put $p=\std(q)$.
	Then \[\std(\ns\nu(q-u))=\std(\ns\nu(p-u))=\nu(p-u)\] by Equation~(\ref{equ::V_fin_inf})
	and thus \[h_q(u)=\std(\ns \nu(q-u)-\ns\nu(q))=\nu(p-u)-\nu(p).\]
	Suppose now that $q\in \ns V\setminus\ns V_\fin$.
	Then $q=q_1+q_2$, with $q_1\in\ns H_{K_q}^\larg$ and $q_2\in\ns V_\fin$.
	We put $p=\std(q_2)$ and $x=\ns\nu(q_1)\gg0$, using Lemma~\ref{CarryOn}.
	For all $k\in K_q$ we have $\ns\xi_k(q_1+q_2-u)=x+\ns\xi_k(q_2-u)$ and thus 
	\[\ns\nu(q-u)=\ns\nu_{K_q}(q_1+q_2-u)=x+\ns\nu_{K_q}(q_2-u).\]
	Similarly,
	$\ns\nu(q)=\ns\nu_{K_q}(q_1+q_2)=x+\ns\nu_{K_q}(q_2)$ and therefore
	\[
	\ns\nu(q-u)-\ns\nu(q)=\ns\nu_{K_q}(q_2-u)-\ns\nu_{K_q}(q_2).
	\]
	Hence
	\[
	h_q(u)=\std(\ns\nu(q-u)-\ns\nu(q))=\std(\ns\nu_{K_q}(q_2-u)-\ns\nu_{K_q}(q_2))=\nu_{K_q}(p-u)-\nu_{K_q}(p).
	\]
	This shows that all horofunctions are as claimed in Theorem~\ref{AsymHoro}.
	
	Conversely, we claim that each of these functions is indeed a horofunction.
	This is clear by definition for the functions 
	\[
	u\longmapsto\nu(p-u)-\nu(p), \text{ for }p\in V.
	\]
	Suppose that $L\subseteq K$ is a dual face. We fix a vector $w\in V$
	with $\nu(w)=1$ such that $L=\{\ell\in K\mid\xi_\ell(w)=1\}$.
	There exists $\eps>0$ such that $\xi_k(w)<1-\eps$ for all 
	$k\in K\setminus L$.
	We choose a nonstandard real $t\gg 0$ and we put $q=tw$.
	Since $t\eps\gg0$, we have $K_q=L$ and 
	$\ns\nu_L(q-u)=t+\ns\nu_L(-u)$ for all $u\in V$, whence
	\[
	h_q(u)=\nu_L(-u).
	\]
	Hence this map is a horofunction. But translation by $-p$ is an isometry in
	$\Isom_\delta(V)$ and thus 
	\[
	u\longmapsto \nu_L(p-u)-\nu_L(p)
	\]
	is also a horofunction, for all $p\in V$.
	This completes the proof of Theorem~\ref{AsymHoro}.
\end{proof}

\begin{Rem}\label{ARemark}
	The previous proof gives us in addition the following. For $p\in \ns V\setminus\ns V_\fin$
	and $u\in V$ we have 
	\begin{equation}\label{h_pHorofunction}
	h_p(u)=\nu_L(q-u)-\nu_L(q),
	\end{equation}
	where $L=K_p$ and $q=\std(p_2)$ in the decomposition $p=p_1+p_2$, with 
	$p_1\in H_{K_p}^\larg$ and $p_2\in \ns V_\fin$ as in Lemma~\ref{CarryOn}.
\end{Rem}

We noted above that the abelian group $V\subseteq\Isom_\delta(V)$ acts on the set of normalized horofunctions. 
To fix some notation, we put 
\[
\tau_w(x)=w+x,
\]
for $x,w\in V$.
Now we calculate the $V$-stabilizers of the normalized horofunctions. 
Since $V$ is abelian and acts transitively on the 
sets 
\[
\{v\longmapsto\nu(p-v)-\nu(p)\mid p\in V\}\quad\text{ and }\quad
\{v\longmapsto\nu_L(p-v)-\nu_L(p)\mid p\in V\},\] 
it suffices to do this for the horofunctions
\[
v\longmapsto\nu(-v)\quad\text{ and }\quad
v\longmapsto\nu_L(-v),
\]
where $L$ is any dual face.
The first horofunction has $0$ as its unique minimum. Hence its $V$-stabilizer is trivial.
To analyze the second case, we put, for $k\in L$,
\[
C_{k,L}=\{v\in V\mid\xi_k(v)>\xi_\ell(v)\text{ for all }\ell\in L\setminus\{k\}\}.
\]

\begin{Lem}
	Let $L$ be a dual face, and $k\in L$. Then $C_{k,L}$ is a nonempty open set.
\end{Lem}
\begin{proof}
	Recall that $A_0,\ldots,A_m$ are the codimension-1-faces of $B$.
	We choose a point $u\in A_k$ such that $\xi_\ell(u)<1$ for all $\ell\in K\setminus\{k\}$.
		Therefore $u\in C_{k,L}$.
	It is clear from the definition that $C_{k,L}$ is open.
\end{proof}
The $C_{k,L}$ are thus nonempty open positive cones
\footnote{A \emph{cone} $C$ in a real vector space $W$ is a subsemigroup $C\subseteq W$, such that $sC\subseteq C$ holds for all $s>0$.}
in the vector space $V$, with 
\[
 V=\bigcup\{C_{k,L}\mid k\in L\}\cup\bigcup\{H_{k,\ell}\mid k,\ell\in L,k\neq \ell\}.
\]
In particular, $\bigcup\{C_{k,L}\mid k\in L\}$ is open and dense in $V$.

\begin{Lem}
	\label{lem::lem_stab_L}
	If $L$ is a dual face, then the $V$-stabilizer of $\nu_L$ is $H_L$.
\end{Lem}
\begin{proof}
	For $w\in V$ we have 
	\[(\tau_{w}\nu_L)(u)=\nu_L(u-w)-\nu_L(-w).\]
	Hence if $w\in H_L$, then $\tau_w\nu_L=\nu_L$. Therefore $H_L$ is contained in
	the stabilizer of $\nu_L$.
	
	Let $w\in V$ and suppose $\tau_w\nu_L=\nu_L$.
	We choose $\ell\in L$ in such a way that $\nu_L(-w)=\xi_\ell(-w)$.
	Let $k\in L$ be arbitrarily. Since $C_{k,L}$ is an nonempty open cone,
	$C_{k,L}\cap (w+C_{k,L})\neq\emptyset$. Hence we may choose
	an element $v\in C_{k,L}$ in such a way that $v-w\in C_{k,L}$.
	Then $\nu_L(v)=\xi_k(v)$ and $\nu_L(v-w)=\xi_k(v-w)$.
	Since 
	\[
	\nu_L(v)=\nu_L(v-w)-\nu_L(-w),
	\]
	we conclude that
	\[
	\xi_k(v)=\xi_k(v)-\xi_k(w)-\xi_\ell(-w),
	\]
	whence $\xi_k(w)=\xi_\ell(w)$. Since $k\in L$ was chosen in an arbitrary way,
	$\xi_k(w)=\xi_\ell(w)$ holds for all $k\in L$, whence $w\in H_L$.
\end{proof}

\begin{Lem}\label{distinct}
	Let $L,L'\subseteq K$ be two different dual faces and let $p\in V$. Then 
	\[
	\nu\neq\tau_p\nu_L\neq\nu_{L'}.
	\]
\end{Lem}
\begin{proof}
	The $V$-stabilizer of $\nu$ is trivial, while the $V$-stabilizer of 
	$\tau_p\nu_L$ is $H_L\neq\{0\}$. This shows the first inequality.
	
	Assume towards a contradiction that $\tau_p\nu_L=\nu_{L'}$.
	For every $k'\in L'$, there exists some $k\in L$ such that 
	$U=(p+C_{k,L})\cap C_{k',L'}$ is nonempty. For $u\in U$ we have 
	\begin{equation}\label{intersect}
	\xi_{k'}(u)=\xi_{k}(u-p)-\nu_L(-p).
	\end{equation}
	Since $U$ is open, Equation (\ref{intersect}) holds for all $u\in V$,
	because two affine hyperplanes in $V\times\RR$ which intersect in a nonempty 
	relatively open set are equal.
	Thus $k=k'$ and hence $L'\subseteq L$, Similarly, we have $L'\supseteq L$.
	This is a contradiction.
\end{proof}

\begin{Lem}\label{NormOnSubspace}
	Let $L$ be a dual face.
	Then there exists a linear subspace $W_L\subseteq V$ such that 
	$V=H_L\oplus W_L$ and such that $\nu_L$ is an asymmetric norm on $W_L$.
\end{Lem}
\begin{proof}
	We need to find a subspace $W_L$ which is a complement of $H_L$, such that 
	$W_L\cap N_L=\{0\}$, where $N_L=\{v\in V\mid\nu_L(v)\leq0\}$ is the negative cone of $\nu_L$. 
	Then $\nu_L$ restricts to an asymmetric norm on $W_L$.
	
	We put $\eta=\sum_{\ell\in L}\xi_\ell$. There exists $u\neq 0$ such that $\xi_\ell(u)=1$
	holds for all $\ell\in L$. Therefore $\eta(u)\neq0$ and thus $\eta\neq0$.
	Since $u\in H_L$, we have $V=H_L+\eta^\perp$. 
	We choose a subspace $W_L\subseteq \eta^\perp$
	such that $V=H_L\oplus W_L$. 
	Suppose that $w\in W_L\cap N_L$. Then $\xi_\ell(w)\leq0$ for all $\ell\in L$. On the other
	hand $\eta(w)=0$, whence $\xi_\ell(w)=0$ for all $\ell\in L$. Thus $w\in H_L$ and therefore
	$w=0$. This shows that $\nu_L$ restricts to an asymmetric norm on $W_L$.
\end{proof}

Combining these results, we can describe the bordification $\widehat V$ of $V$ now as a stratified space.
Recall that  $\Sigma=\{L\subseteq K\mid L\text{ is a dual face}\}$.
We put \[V_L=V/H_L\text{ for }L\in\Sigma\text{ and }V_K=V.\]
\begin{Thm}\label{Stratification}
	There is a $V$-equivariant bijection
	\[
	\Phi:\widehat V\longrightarrow \bigsqcup\{V_L\mid L\in\Sigma\cup\{K\}\}
	\]
	given by
	\[
	\Phi[v\longmapsto \nu(p-v)-\nu(p)]=p
	\]
	and 
	\[
	\Phi[v\longmapsto \nu_L(p-v)-\nu_L(p)]=p+H_L
	\]
	The restriction of $\Phi^{-1}$ to each of the vector spaces $V_L$ is a homeomorphism.
\end{Thm}

On the right-hand side we have to use the disjoint union since it may happen that $H_L=H_{L'}$
holds for different dual faces $L,L'$, eg.~if $B$ is a cube.

\begin{proof}
	By Lemma~\ref{distinct}, the map $\Phi$ is well-defined and surjective.
	By Lemma~\ref{lem::lem_stab_L}, it is also injective.
	The $V$-stabilizer of the map $[v\longmapsto \nu(p-v)-\nu(p)]$ is trivial
	and the $V$-stabilizer of the map $[v\longmapsto \nu_L(p-v)-\nu_L(p)]$ is $H_L$
	by Lemma~\ref{lem::lem_stab_L}. Hence $\Phi$ is an equivariant bijection.
	
	For the horofunctions $v\longmapsto \nu(p-v)-\nu(p)$, the map $\Phi$ is just the inverse of the 
	topological embedding $V\longrightarrow \widehat V$ and therefore a homeomorphism.
	
	Assume now that $L\subseteq K$ is a dual face and put 
	$\phi_p(v)=\nu_L(p-v)-\nu_L(p)$. Let $W_L\subseteq V$ be as in Lemma~\ref{NormOnSubspace}.
	The map $p\longmapsto \phi_p$ is a continuous map $W_L\longrightarrow \widehat V\subseteq I_{\Bnd,0}(V)$
	by Lemma~\ref{TriangleInEq}.
	If we combine it with the restriction map $I_{\Bnd,0}(V)\longrightarrow I_{\Bnd|W_L,0}(W_L)$,
	we obtain an embedding $W_L\longmapsto \widehat{W_L}$. Therefore the map 
	$W_L\longrightarrow \widehat V$ is also an embedding. Now there is an isomorphism of topological vector spaces
	$W_L\longrightarrow V_L=V/H_L$ and thus $\Phi^{-1}$ is a homeomorphism on $V_L$.
\end{proof}

The description of the horofunctions in Theorem~\ref{AsymHoro} allows us also to describe
the horofunctions using rays in $V$.

\begin{Def}
	Let $L\subseteq K$ be a dual face. We put 
	\[
	H_L^+=\{v\in H_L\mid \xi_\ell(v)>\xi_k(v)\text{ for all }\ell\in L\text{ and }k\in K\setminus L\}.
	\]
	From the definition of a dual face we see that $H_L^+\neq\emptyset$.
	Thus $H_L^+$ is a nonempty open cone in $H_L$.
	For formal reasons it will be convenient to put 
	\[
	H_K^+=H_K=\{0\}.
	\]
\end{Def}

\begin{Lem}\label{BusemannLemma}
	Let $L$ be a dual face and assume that $w\in H_L^+$.
	Then the family of functions
	$(\tau_{-tw}\nu)_{t>0}$ converges in $I_{\Bnd,0}(V)$ to $\nu_L$ as $t$ gets large.
\end{Lem}
\begin{proof}
	There exists $\eps>0$ such that $\xi_\ell(v)>\xi_k(v)$ holds for all
	$\ell\in L$, $k\in K\setminus L$ and $v\in B_\eps(w)$.
	Hence $\tau_{-tw}\nu$ and $\tau_{-tw}\nu_L=\nu_L$ agree on the ball $B_{t\eps}(0)$, for $t>0$.
	As $t$ grows, this ball becomes arbitrarily large.
\end{proof}
We note also the following. If $w\in V$ is a nonzero vector, then the set 
\[
K(w)=\{\ell\in K\mid \xi_\ell(w)=\nu(w)\}
\]
is a dual face
\footnote{The relation between $K_w$ and $K(w)$ is as follows. If 
	$t\gg0$ is a nonstandard real, then $K_{tw}=K(w)$.}, 
and $w\in H_{K(w)}^+$. Moreover 
$\{u\in H_L^+\mid\nu(u)=1\}$ is an open face of the polyhedron
$B=\{u\in V\mid\nu(u)\leq1\}$. In particular, 
\[
V\setminus\{0\}=\bigsqcup\{H_L^+\mid L\subseteq K\text{ is a dual face}\}.
\]
The sets $H_{\{\ell\}}^+$, for $\ell\in K$,
are pairwise disjoint open cones in $V$, and their union is dense in $V$.

\begin{Prop}\label{BusemannProp}
	Let $p,w\in V$. Then the family of normalized horofunctions
	\[
	u\longmapsto\nu(p+tw-u)-\nu(p+tw)
	\]
	converges to the normalized horofunction $u\longmapsto \nu_{K(w)}(p-u)-\nu_{K(w)}(p)$ as $t$ gets large,
	where $K(w)=\{\ell\in K\mid \xi_\ell(w)=\nu(w)\}$.
	In particular, every normalized horofunction arises as such a limit along an affine line in $V$.
\end{Prop}
\begin{proof}
	This follows from Lemma~\ref{BusemannLemma} and the remark preceding this proposition.
\end{proof}
Now we improve on Theorem~\ref{Stratification} by describing the topology on the right-hand side.
\begin{Def}
      Let $L$ be a dual face. Given $\eps>0$ and $q\in V$, we put $D=q+B_\eps(0)+H_L^+$ and 
      \[
      U(L,\eps,q)=D\sqcup\bigsqcup\{(D+H_{L'})/H_{L'}\mid L'\in\Sigma\text{ with }L'\supseteq L\}
      \subseteq \bigsqcup\{V_{L'}\mid L'\in\Sigma\cup\{K\}\}.
      \]
      We put also
      \[
      U(K,\eps,q)=q+B_\eps(0)\subseteq V \subseteq \bigsqcup\{V_{L'}\mid L'\in\Sigma\cup\{K\}\}.
      \]
\end{Def}

      We note that the collection of these sets is invariant under the action of the group $V$ by translations.
      Now we show that these sets form a basis for the topology imposed on $\bigsqcup\{V_{L'}\mid L'\in\Sigma\cup\{K\}\}$
      by the bijection $\Phi$
      in Theorem~\ref{Stratification}.

\begin{Lem}\label{Nbhd}
	Let $L\in\Sigma\cup\{K\}$. Given real numbers $r,s>0$ there exist $\eps>0$ and $q\in H_L^+$ such that 
	\[
	|\nu_L(-u)-\phi(u)|<s
	\]
	holds for every $u\in \bar B_r(0)$ and every normalized horofunction $\phi$ with $\Phi(\phi)\in U(L,\eps,q)$.
\end{Lem}
\begin{proof}
	We put $\eps=\frac{s}{2\gamma}$, where $\gamma>0$ is as in Lemma~\ref{TriangleInEq}. 
	
	Suppose that $K=L$. Then $ |\nu(-u)-\nu(p-u)+\nu(p)|<s$
	holds for every $p\in U(K,\eps,0)=B_\eps(0)$.
	
	Now suppose that $K\neq L$.
	Then we choose $q\in H_L^+$ in such a way that 
	for all $u\in \bar B_{r+\eps}(q)$ we have $(\xi_\ell-\xi_k)(u)>0$ whenever $\ell\in L$ and $k\in K\setminus L$.
	\footnote{\label{fn}This is possible because every $w\in H_L^+$ has a small neighborhood such that for every $u$ in this neighborhood,
		$(\xi_\ell-\xi_k)(u)>0$, for $k,\ell$ as above. Then we multiply $w$ by a large real number to obtain~$q$.}
	Let $p_1\in B_\eps(0)$ and $p_2\in H_L^+$. If $L'\subseteq K$ is any subset containing $L$
	and if $u\in\bar B_r(0)$, we have
	\begin{multline*}
	|\nu_L(-u)-\nu_{L'}(q+p_1+p_2-u)+\nu_{L'}(q+p_1+p_2)| \\
	=|\nu_L(-u)-\nu_L(q+p_1+p_2-u)+\nu_L(q+p_1+p_2)|\\
	= |\nu_L(-u)-\nu_L(p_1-u)+\nu_L(p_1)|<2\gamma\eps=s.
	\end{multline*}
	The claim follows.
\end{proof}

\begin{Prop}\label{Open}
	Let $L\in\Sigma\cup\{K\}$, let $\eps>0$ and $q\in V$. Then 
	the set $\Phi^{-1}(U(L,\eps,q))\subseteq\widehat V$ is open.
\end{Prop}
\begin{proof}
	We have to show that for every normalized horofunction $\phi\in\Phi^{-1}(U(L,\eps,q))$ there exist 
	$r,s>0$ such that every normalized horofunction $\psi$ with 
	\[
	|\phi(u)-\psi(u)|<s\text{ for all }u\in\bar B_r(0)
	\]
	is contained in the set $\Phi^{-1}(U(L,\eps,q))$. From the definition of $U(L,\eps,q)$, we may
	write \[\phi(u)=\nu_{L'}(q+q_1+q_2-u)-\nu_{L'}(q+q_1+q_2),\] with 
	$q_1\in B_\eps(0)$ and $q_2\in H_L^+$ and $L'\supseteq L$ a dual face, or $L'=K$.
	We put \[D=q+B_\eps(0)+H_L^+.\]
	If $L'\supseteq L$ is a dual face or if $L'=K$, then 
	\begin{equation}
	\label{InclusionCones}
	H_{L'}^+\subseteq\overline{H_L^+}\subseteq B_\eps(0)+H_L^+.
	\end{equation}
	We argue by contradiction, using again the ultrapower.
	
	Suppose that the claim is false. Then we find for every pair of natural numbers $(m,n)$
	a counterexample, that is,
	a normalized horofunction \[\psi_{m,n}(u)=\nu_{L_{m,n}}(p_{m,n}-u)-\nu_{L_{m,n}}(p_{m,n}),\]
	which satisfies 
	\[
	|\phi(u)-\psi_{m,n}(u)|<2^{-m}\text{ for all }u\in\bar B_{2^n}(0),
	\]
	and which is not in $U(L,\eps,q)$. We note also that then
	\[
	|\phi(u)-\psi_{m,n}(u)|<2^{-m'}\text{ for all }u\in\bar B_{2^{n'}}(0)
	\]
	holds for all $m'\leq m$ and all $n'\leq n$.
	The $\omega_1$-saturation of the ultrapower gives us therefore an $L''\in\Sigma\cup\{K\}$,
	a point $p\in \ns V$,
	and nonstandard reals $r,s>0$ with $r\gg0$ and $s\in\ns\RR_\infi$, such that 
	\begin{equation}
	\label{infinitesimal}
	\ns|\ns\phi(u)-\ns\nu_{L''}(p-u)+\ns\nu_{L''}(p)|<s 
	\end{equation}
	holds for all $u\in \ns B_r(0)$.
	The map $\ns V\longrightarrow\ns \RR$ given by
	$v\longmapsto\ns \nu_{L''}(p-v)-\ns\nu_{L''}(p)$ is not contained in
	$\ns U(L,\eps,q)$.
	In case $L''\supseteq L'$, we have therefore necessarily $p\not\in \ns D+\ns H_{L''}$.
	The Inequality~(\ref{infinitesimal}) shows that for all $u\in V$ we have 
	\[
	\phi(u)=\std(\ns\nu_{L''}(p-u)-\ns\nu_{L''}(p)),
	\]
	because $r\gg0$ and $\std(s)=0$.
	We distinguish three cases.
	
	\smallskip Case (i).
	If $p\in\ns V_\fin$ we put $w=\std(p)$. Then 
	\[
	\nu_{L'}(q+q_1+q_2-u)-\nu_{L'}(q+q_1+q_2)=\phi(u)=\nu_{L''}(w-u)-\nu_{L''}(w)
	\]
	holds for all $u\in V$.
	But then $L'=L''$ by Lemma~\ref{distinct}, and $q+q_1+q_2+v=w$ for some $v\in H_{L'}$
	by Lemma~\ref{lem::lem_stab_L}. Thus $w$ is contained in the open set $D+H_{L'}$. 
	This set is open and contains therefore a small $\eps'$-neighborhood of $w$, for some real $\eps'>0$.
	Hence $\ns D+\ns H_{L'}$ contains $p$, because $\ns d(w,p)<\eps'$ holds for every positive real
	$\eps'>0$. We have arrived at a contradiction.
	
	\smallskip Case (ii).
	Suppose next that $p\in\ns V\setminus\ns V_\fin$ and that $L''=K$. For $u\in V$ we have then 
	\[\std(\ns\nu_K(p-u)-\ns\nu_K(p))=h_p=\phi.\] 
	We decompose $p$ as $p=p_1+p_2$, with $p_1\in \ns H_{K_p}^\larg$ and $p_2\in \ns V_\fin$ as in Lemma~\ref{CarryOn}.
	If we put $w=\std(p_2)$, then 
	\[
	h_p(u)=\nu_{K_p}(w-u)-\nu_{K_p}(w).
	\]
	Therefore $L'=K_p$ and $w+v=q+q_1+q_2$ for some $v\in H_{L'}$ by Lemma~\ref{distinct} and Lemma~\ref{lem::lem_stab_L}.
	We have \[p_1-v\in \ns H_{L'}^\larg\subseteq \ns H_{L'}^+\subseteq \ns(\overline{H_L})\]
	by Equation~\ref{InclusionCones}, and the right-hand side is a subsemigroup of $\ns V$.
	As in the previous case, we have also \[p_2+v\in \ns D= q+\ns B_\eps(0)+\ns H_L^+=q+\ns B_\eps(0)+\ns(\overline{H_L^+}).\]
	Thus \[p=p_1+p_2\in q+ \ns B_\eps(0)+\ns(\overline{H_L^+})=\ns D.\]
	Again, we have arrived at a contradiction.
	
	\smallskip Case (iii).
	Suppose finally that $p\in\ns V\setminus\ns V_\fin$ and that $L''\subsetneq K$. 
	We choose $p'\in\ns H_{L''}^+$ in such a way that $\xi_\ell(p'+p+u)\gg\xi_k(p'+p+u)$ holds for all 
	$\ell\in L''$, all $k\in K\setminus L''$ and all $u\in\ns V_\fin$. $^{\ref{fn}}$ 
	Thus $K_{p+p'}\subseteq L''$,
	and
	\[
	\ns\nu_{L''}(p-u)-\ns\nu_{L''}(p)=\ns\nu_{L''}(p+p'-u)-\ns\nu_{L''}(p+p')=\ns\nu(p+p'-u)-\ns\nu(p+p')
	\]
	holds for all $u\in\ns V_\fin$. Therefore $K_{p+p'}=L'$ by Remark~\ref{ARemark}, and thus $L''\supseteq L'\supseteq L$.
	We decompose $p+p'=p_1+p_2$, with $p_1\in \ns H_{L'}^+$ and $p_2\in \ns V_\fin$ as in Lemma~\ref{CarryOn}, and we put $w=\std(p_2)$. 
	Then $w+v=q+q_1+q_2$ for some $v\in H_L'$ and hence
	\[p_2+v\in \ns D=q+\ns B_\eps(0)+\ns(\overline{H_L^+})\] as in the previous case.
	Moreover, $p_1-v\in \ns H_{L'}^+\subseteq {\ns (\overline{H_L^+})}.$
	Thus $p+p'\in \ns D$ and therefore $p\in \ns D+H_{L''}$. Again, this is a contradiction. This last case finishes the proof.
\end{proof}

\begin{Thm}\label{TopologyBasis}
	The sets $U(L,\eps,q)$, for $\eps>0$, $q\in V$ and $L\in\Sigma\cup\{K\}$,
	form a basis for the topology imposed on $\bigsqcup\{V_L\mid L\in\Sigma\cup\{K\}\}$
	by the bijection $\Phi$.
\end{Thm}
\begin{proof}
	By Proposition~\ref{Open} the sets $U(L,q,\eps)$ are open and
	by Lemma~\ref{Nbhd} the sets containing a given point form a 
	neighborhood basis of this point.
\end{proof}

\begin{Cor}\label{Corol1}
      Let $p\in V$, Then the set 
      \[
      \{U(L,\eps,p+q)\mid q\in H_L^+\text{ and }\eps>0\}
      \]
      is a neighborhood basis of the point $p+H_L\in V_L$ in $\bigsqcup\{V_L\mid L\in\Sigma\cup\{K\}\}$, in the topology imposed by $\Phi$.
\end{Cor}
\begin{proof}
      Assume first that $p=0$. Then $0\in q+B_\eps(0)+H_L^+ +H_L=B_\eps(0)+H_L$, hence each of these open sets contains the point $H_L\in V_L$.
      By Lemma~\ref{Nbhd}, these sets form a neighborhood basis of the point. The general claim follows now by translation by $p$.
\end{proof}

\begin{Cor}\label{Corol2}
	Let $(p_n)_{n\in\NN}$ be a sequence in $V$. Then the sequence of normalized horofunctions
	\[
	u\longmapsto \nu(p_n-u)-\nu(p_n)
	\]
	converges to the normalized horofunction
	\[
	u\longmapsto \nu_L(p-u)-\nu(p),
	\]
	for $L\in\Sigma\cup\{K\}$ and $p\in V$ if and only if for each
	$q\in H_L^+$, we have 
	\[
	\lim_{n\to\infty} d(p_n,p+q+H_L^+)=0.
	\]
\end{Cor}
\begin{proof}
	{We have $U(L,\eps,p+q)\cap V=p+q+B_\eps(0)+H_L^+$.}
\end{proof}

This yields in particular another proof of Proposition~\ref{BusemannProp}.

\begin{Cor}\label{Corol3}
	For $L\in\Sigma\cup\{K\}$, the closure of $V_L$ in $\bigsqcup\{V_L\mid L\in\Sigma\cup\{K\}\}$, in the topology imposed by $\Phi$, is 
	$\bigsqcup\{V_{L'}\mid L'\in\Sigma\cup\{K\}\text{ with }L\supseteq L'\}$.
\end{Cor}
The combinatorial structure of the stratification of $\widehat V$ in Theorem~\ref{Stratification}, with respect to the closure operation, 	
is therefore poset-isomorphic to the poset $(\Sigma\cup\{K\},\subseteq)$.
This poset, in turn, is anti-isomorphic to the poset of all proper faces of $B$, including the empty face.

Now we show that there is a homeomorphism between $\widehat V$ and the dual polyhedron $B^\vee$ of $B$.
We use generalized moment maps, similarly to \cite{Fulton,JiSchilling}.
To construct such a homeomorphism we define auxiliary maps.
For $L\in\Sigma\cup\{K\}$ we put 
\begin{align*}
	a_L(p)&=\sum_{k\in L}\exp(\xi_k(p))\xi_k\\
	b_L(p)&=\sum_{k\in L}\exp(\xi_k(p))\\
	c_L(p)&=\frac{a_L}{b_L}.
\end{align*}
\begin{Rem}\label{Taylor}
	We note the following.
\begin{enumerate}[(i)]
	\item The Taylor expansion of $b_L(p+tv)$ at the point $p$ is
	\[\textstyle
	b_L(p+tv)=\sum_{k\in L}\exp(\xi_k(p))(1+t\xi_k(v)+\frac{1}{2}t^2\xi_k(v)^2+\cdots).
	\]
	Hence the derivative of $b_L$ at $p$ is \[Db_L(p)(v)=a_L(p)(v)\] and the Hessian of $b_L$ at $p$ is the quadratic form
	\[Hb_L(p)(v)=\sum_{k\in L}\exp(\xi_k(p))\xi_k(v)^2.\]
	\item Therefore $c_L(p)$ is the derivative of the map \[f_L(p)=\log(b_L(p)).\]
	\item The image of $c_L$ is contained in the convex hull $B_L^\vee$ of $\{\xi_k\mid k\in L\}$, which is a face in the dual polyhedron $B^\vee\subseteq V^\vee$.
	\item If $v\in H_L$, then $c_L(p+v)=c_L(p)$, because $a_L(p+v)=\exp(\xi_L(v))a_L(p)$ and $b_L(p+v)=\exp(\xi_\ell(v))b_L(v)$,
	for any choice of $\ell\in L$.
	\item Since $V\longrightarrow V_L$ is an open map, the map $V_L\longmapsto V^\vee$, $p+H_L\longmapsto c_L(p)$ is continuous.
\end{enumerate}
\end{Rem}
\begin{Lem}\label{Injective}
	Let $L\in\Sigma\cup\{K\}$. Then the map $p+H_L\longmapsto c_L(p)$ is injective and open on $V_L$.
\end{Lem}
\begin{proof}
	Let $W_L\subseteq V$ be a linear subspace such that $V=W_L\oplus H_L$.
	We claim that 
	\[
	(c_L(q)-c_L(p))(q-p)>0
	\]
	holds for all $p,q\in W_L$ with $p\neq q$. This will clearly show that $c_L$ is injective on $W_L$.
	The Hessian of the map $f_L$ is 
	\[
	Hf_L(p)(v)=\frac{b_L(p)Hb_L(p)(v)-Db_L(p)(v)^2}{b_L(p)^2}.
	\]
	We claim that this quadratic form is positive definite.
	We put $e_k=\exp(\xi_k(p))$ for short, and
	we have to show that 
	\[
	\sum_{k,\ell\in L}e_k e_\ell\xi_k(v)^2>\sum_{k,\ell\in L}e_k e_\ell\xi_k(v)\xi_\ell(v)
	\]
	holds for all $v\neq 0$. Equivalently, we have to show for $v\neq 0$ that 
	\begin{equation}\label{Young}
	\sum_{(k,\ell)\in M}e_k e_\ell(\xi_\ell(v)^2+\xi_k(v)^2)>
	\sum_{(k,\ell)\in M} 2 e_k e_\ell\xi_\ell(v)\xi_k(v),
	\end{equation}
	where $M=\{(k,\ell)\in L\times L\mid k<\ell\}$.
	Young's Inequality says that $x^2+y^2\geq 2xy$, with equality if and only if $x=y$.
	Hence the left-hand side of Inequality~(\ref{Young}) is not smaller than the right-hand side.
	If we would have equality, then we would have $\xi_k(v)=\xi_\ell(v)$ for all $k,\ell\in L$ and thus $v=0$.
	Therefore the Hessian of $f_L$ is positive definite. This implies by convexity that
	\[
	(Df_L(q)-Df_L(p))(q-p)>0
	\]
	holds for all $p,q\in W_L$ with $p\neq q$, see eg.~\cite{RW98}~Thm.~2.14.
	This follows also directly, since
	\[
	(Df_L(q)-Df_L(p))(q-p)=\int_0^1Hf_L((1-t)p+tq)(q-p)dt.
	 \]
	Since $Hf_L$ is positive definite, the derivative $Dc_L(p)$ of $c_L$ has rank $\dim(W_L)=\dim(V_L)$ at every point $p\in W_L$.
	Hence $c_L$ is an open map on $W_L\cong V_L$. 
\end{proof}

\begin{Def}
	We define a map 
	\[
	c:\bigsqcup\{V_L\mid L\in\Sigma\cup\{K\}\}\longmapsto B^\vee
	\]
	by putting
	\[
	c(p+H_L)=c_L(p)\text{ for }L\in\Sigma\text{ and } c(p)=c_K(p)\text{ for }p\in V.
	\]
	\end{Def}
\begin{Lem}\label{IsContinuous}
	The map $c$ is continuous.
\end{Lem}
\begin{proof}
	Let $q_0\in V$ and $L\in\Sigma\cup\{K\}$. 
	We show that $c$ is continuous at the point \[q_0+H_L\in\bigsqcup\{V_L\mid L\in\Sigma\cup\{K\}\}.\]
	To this end, we show that given a real number $s>0$,
	we can choose a neighborhood $U(L,\eps,q_0+q)$ of $q_0+H_L$ in such a way
	that $||c(p+H_{L'})-c(q_0+H_L)||<s$ holds for all $p+H_{L'}\in U(L,\eps,q_0+q)$.
	Here $||.||$ is the dual Euclidean norm on $V^\vee$ determined by the Euclidean norm $||.||$ on $V$.

	We choose $\eps>0$ in such a way that 
	\begin{align}
	\label{i1} \frac{||a_L(q_0+q_1)-a_L(q_0)||b_L(q_0)}{b_L(q_0)^2}    & <\frac{s}{8} \\
	\label{i2} \frac{||a_L(q_0)||\,|b_L(q_0+q_1)-b_L(q_0)|}{b_L(q_0)^2} & <\frac{s}{8} \\
	\intertext{and}
	\label{i3} b_L(q_0) & <2b_L(q_0+q_1)
	\intertext{hold for all $q_1\in B_\eps(0)$. Then we choose $\ell\in L$ and $q\in H_L^+$ in such a way that}
	\label{i4}
	\frac{\mu}{b_L(q_0)^2}\sum_{j\in K}\exp(\xi_j(q_0+q_1)-(\xi_\ell-\xi_j)(q)) & \leq \frac{s}{8}
	\end{align}
	holds for all $q_1\in B_\eps(0)$, where
	\[\textstyle \mu=\max\{||a_L(q_0)||,||\xi_1||b_L(q_0),\cdots,||\xi_m||b_L(q_0)\}.\]

	Suppose that $p+H_{L'}\in U(L,\eps,q_0+q)$. Then $L'\supseteq L$ and \[p+v=q_0+q_1+q_2+q,\] for some $v\in H_{L'}$, $q_1\in B_\eps(0)$ and 
	$q_2\in H_L^+$.
	We compute
	\begin{align*}
	b_{L'}(p+v) & = \sum_{k\in L}\exp(\xi_k(q_0+q_1+q_2+q))+\sum_{j\in L'\setminus L}\exp(\xi_j(q_0+q_1+q_2+q))\\
	& =\exp(\xi_\ell(q_2+q))\bigg(\sum_{k\in L}\exp(\xi_k(q_0+q_1))+\sum_{j\in L'\setminus L}\exp(\xi_j(q_0+q_1)-(\xi_\ell-\xi_j)(q_2+q))\bigg)\\
        &=\exp(\xi_\ell(q_2+q))\bigg(b_L(q_0+q_1)+\sum_{j\in L'\setminus L}\exp(\xi_j(q_0+q_1)-(\xi_\ell-\xi_j)(q_2+q))\bigg).
	\end{align*}
	We expand $a_{L'}$ similarly and obtain
	\[
	c_{L'}(p)=c_{L'}(p+v)=\frac{a_L(q_0+q_1)+a'}{b_L(q_0+q_1)+b'},
	\]
	with 
	\[
	a'=\sum_{j\in L'\setminus L}\exp(\xi_j(q_0+q_1)-(\xi_\ell-\xi_j)(q_2+q))\xi_j
	\]
	and
	\[
	b'=\sum_{j\in L'\setminus L}\exp(\xi_j(q_0+q_1)-(\xi_\ell-\xi_j)(q_2+q)).
	\]
	We note that 
	\[ \frac{||a'||b_L(q_0)}{b_L(q_0)^2}  <\frac{s}{8}\text{  and } \frac{||a_L(q_0)||b'}{b_L(q_0)^2}  <\frac{s}{8} \] by Inequality~(\ref{i4})
	and that
	\[\frac{1}{2}b_L(q_0)^2 \leq b_L(q_0)(b_L(q_0+q_1)+b')\] by Inequality~(\ref{i3}).
	Hence we have, by the Inequalities~(\ref{i1}) and (\ref{i2}),
	\begin{align*}
	\big\|c_{L'}(p)-c_L(q_0)\big\| & = \bigg\|\frac{a_L(q_0+q_1)+a'}{b_L(q_0+q_1)+b'}-\frac{a_L(q_0)}{b_L(q_0)}\bigg\| \\
	& = \frac{||(a_L(q_0+q_1)+a')b_L(q_0)-a_L(q_0)(b_L(q_0+q_1)+b')||}{b_L(q_0)(b_L(q_0+q_1)+b')} \\
	& \leq 2\frac{||(a_L(q_0+q_1)+a')b_L(q_0)-a_L(q_0)(b_L(q_0+q_1)+b')||}{b_L(q_0)^2} \\
	& \leq 2\frac{||(a_L(q_0+q_1)+a'-a_L(q_0)||b_L(q_0)}{b_L(q_0)^2} \\
	&\qquad{}+2\frac{||a_L(q_0)||(b_L(q_0+q_1)+b'-b_L(q_0))}{b_L(q_0)^2} \\
	& \leq 2\frac{(||(a_L(q_0+q_1)-a_L(q_0)||b_L(q_0)}{b_L(q_0)^2} + 2\frac{||a'||b_L(q_0)}{b_L(q_0)^2} \\
	&\qquad{}+2\frac{||a_L(q_0)||(|b_L(q_0+q_1)-b_L(q_0)|)}{b_L(q_0)^2} + 2\frac{||a_L(q_0)||b'}{b_L(q_0)^2}\\
	&\leq s.
	\qedhere
	\end{align*}
\end{proof}
We need at this stage a topological result.
\begin{Lem}\label{Topo}
	Let $f:(X,A)\longrightarrow (Y,B)$ be a continuous map of compact topological pairs. Assume  
	$A=f^{-1}(B)$ and that the restriction
	$f:X\setminus A\longrightarrow Y\setminus B$ is injective and open. If $X\setminus A$ and $Y\setminus B$
	are homeomorphic to $\RR^n$, then the restriction $f:X\setminus A\longrightarrow Y\setminus B$ is surjective.
\end{Lem}
\begin{proof}
	We consider the induced map $\bar f:X/A\longrightarrow Y/B$. Both spaces are compact
	(they are Hausdorff since $X$ and $Y$ are regular) and may therefore be identified with
	the Alexandrov compactifications of $X\setminus A$ and $Y\setminus B$, respectively.
	Hence $X/A\cong\mathbb{S}^n\cong Y/B$. It suffices to show that $\bar f$ is surjective.
	Let $p\in X\setminus A$ and $q=f(p)$. Since $f$ is a homeomorphism near $p$, we obtain  
	by excision in singular homology an isomorphism
	$\bar f_*:H_n(X/A,(X\setminus\{p\})/A)\longrightarrow H_n(Y/B,(Y\setminus\{q\})/B)$.
	From the long exact homology sequence we obtain an
	isomorphism $\bar f_*:H_n(X/A)\longrightarrow H_n(Y/B)$. Thus $\bar f$ has degree $\pm1$ and is therefore surjective.
	Indeed, if $\bar f$ was not surjective, then
	$\bar f$ would factor through a map $X/A\longrightarrow Y/B\setminus\{y\}\longrightarrow Y/B$, for some $y\in Y/B$.
	But $Y/B\setminus\{y\}$ is contractible, whence $\bar f_*=0$. On the other hand, $\bar f_*\neq 0$
	because $H_n(X/A)\cong\ZZ$. Hence $\bar f$ is surjective.
\end{proof}
For $L\in\Sigma\cup\{K\}$ we let $B_L^\vee$ denote the face of $B^\vee$ whose vertex set is $\{\xi_\ell\mid \ell\in L\}$,
and $U_L\subseteq B_L^\vee$ the corresponding open face.
\begin{Thm}\label{DualityThm}
	The map $c$ is a homeomorphism between $\widehat V$ and $B^\vee$ that maps 
	$V_L$ homeomorphically onto $U_L$, for each $L\in\Sigma\cup\{K\}$.
\end{Thm}
\begin{proof}
	The restriction of the continuous map $c$ to any stratum $V_L$ is injective by Lemma~\ref{Injective}.
	Since $c_L$ is an open map, $c(V_L)$ is contained in the open face $U_L\subseteq B_L^\vee$.
	These open faces partition $B^\vee$ and thus $c$ is injective.
	Given $L\in\Sigma\cup\{K\}$, let $A$ denote the union of all $V_{L'}$ with $L'\subsetneq L$,
	and put $X=V_L\cup A$. Then $(X,A)$ is a compact pair. Let $M$ denote the union of all proper faces
	of the face $B^\vee_L$.
	Then $(B^\vee_L,M)$ is also a compact pair, and $c$ restricts to a map of pairs $f:(X,A)\longmapsto(B^\vee_L,M)$.
	The assumptions of Lemma~\ref{Topo} are satisfied and thus $f$ is surjective.
	
	Hence $c$ is surjective. Being a continuous bijection between compact spaces, it is a homeomorphism.
\end{proof}

\end{document}